\documentclass[a4paper]{article}

\makeatletter
\RequirePackage{ifthen}
\newcommand{\Draft}{1}
	\AtBeginDocument{\@ifpackageloaded{showkeys}
	{\renewcommand{\Draft}{1}}
	{\renewcommand{\Draft}{0}}
}
\makeatother

\usepackage{amsmath}
\usepackage{amssymb}
\usepackage{amsthm}
\usepackage{bm}
\usepackage{dsfont}
\usepackage{enumerate}
\usepackage{epic}
\usepackage{eepic}
\usepackage[OT2,OT1]{fontenc}
\usepackage{ifthen}
\usepackage{pifont}
\usepackage{rotating}
\usepackage{stmaryrd}
\usepackage{textcomp}
\usepackage{theoremref}
\usepackage{todonotes}
\usepackage{url}

\newcommand{\bb}[1]{{\mathbb{#1}}}

\newcommand\cyr{%
\renewcommand\rmdefault{wncyr}%
\renewcommand\sfdefault{wncyss}%
\renewcommand\encodingdefault{OT2}%
\normalfont
\selectfont}
\DeclareTextFontCommand{\textcyr}{\cyr}

\DeclareMathOperator{\tr}{tr}

\newlength{\maxlabwidth}

\numberwithin{equation}{section}
\swapnumbers
\theoremstyle{plain}
	\newtheorem{lemma}{Lemma}[section]
	\newtheorem{proposition}[lemma]{Proposition}
	\newtheorem{theorem}[lemma]{Theorem}
	\newtheorem{corollary}[lemma]{Corollary}
	\newtheorem{ntheoreM}[lemma]{}
\theoremstyle{definition}
	\newtheorem{definitioN}[lemma]{Definition}
	\newtheorem{ndefinitioN}[lemma]{}
\theoremstyle{remark}
	\newtheorem{remarK}[lemma]{Remark}
	\newtheorem{examplE}[lemma]{Example}
	\newtheorem{nremarK}[lemma]{}
\newcommand{\thlab}[1]{\thlabel{#1}\label{#1 }}
\renewcommand{\qedsymbol}{\raisebox{-2pt}{\large\ding{113}}}
\newcommand{\defendsymbol}{$\lozenge$}
\newcommand{\qedsymbolsave}{\qedsymbol}

\newenvironment{definition}{\begin{definitioN}}{
	\renewcommand{\qedsymbolsave}{\qedsymbol}\renewcommand{\qedsymbol}{\defendsymbol}
	\popQED{\qed}\renewcommand{\qedsymbol}{\qedsymbolsave}\end{definitioN}}
\newenvironment{ndefinition}[1]{\begin{ndefinitioN}\textbf{\!{#1:}\ }}{
	\renewcommand{\qedsymbolsave}{\qedsymbol}\renewcommand{\qedsymbol}{\defendsymbol}
	\popQED{\qed}\renewcommand{\qedsymbol}{\qedsymbolsave}\end{ndefinitioN}}
\newenvironment{remark}{\begin{remarK}}{
	\renewcommand{\qedsymbolsave}{\qedsymbol}\renewcommand{\qedsymbol}{\defendsymbol}
	\popQED{\qed}\renewcommand{\qedsymbol}{\qedsymbolsave}\end{remarK}}
\newenvironment{example}{\begin{examplE}}{
	\renewcommand{\qedsymbolsave}{\qedsymbol}\renewcommand{\qedsymbol}{\defendsymbol}
	\popQED{\qed}\renewcommand{\qedsymbol}{\qedsymbolsave}\end{examplE}}

\DeclareMathOperator{\Cov}{Cov}

\begin{document}

\ifthenelse{\Draft=1}{\pagenumbering{roman}}

\begin{flushleft}
	{\Large\bf Bounds on order of indeterminate moment\\[2mm] sequences}
	\\[5mm]
	\textsc{
	Raphael Pruckner
	\,\ $\ast$\,\ 
	Roman Romanov
	\,\ $\ast$\,\ 
	Harald Woracek
		\hspace*{-14pt}
		\renewcommand{\thefootnote}{\fnsymbol{footnote}}
		\setcounter{footnote}{2}
		\footnote{This work was supported by a joint project of the Austrian Science Fund (FWF, I\,1536--N25) 
			and the Russian Foundation for Basic Research (RFBR, 13-01-91002-ANF).}
		\renewcommand{\thefootnote}{\arabic{footnote}}
		\setcounter{footnote}{0}
	}
	\\[6mm]
	{\small
	\textbf{Abstract:}
	We investigate the order $\rho$ of the four entire functions in the Nevanlinna matrix of an indeterminate
  Hamburger moment sequence. We give an upper estimate for $\rho$ which is explicit in terms of the parameters of the canonical system associated with
  the moment sequence via its three-term recurrence.
  Under a weak regularity assumption this estimate coincides with a lower estimate, and hence $\rho$ becomes computable.
  Dropping the regularity assumption leads to examples where upper and lower bounds do not coincide and differ from the order.
  In particular we provide examples for which the order is different from its lower estimate due to M.S.Liv\v{s}ic.
	\\[3mm]
	}
\end{flushleft}

\ifthenelse{\Draft=1}{
\fbox{
\parbox{100mm}{
\hspace*{0pt}\\
\centerline{{\Large\ding{45}}\quad\,{\large\sc Draft}\quad{\Large\ding{45}}}
\hspace*{0pt}\\[-3mm]
In final version:\qquad \ding{233} disable package \textsf{showkeys}.\\
Download latest version of bibtex-database-woracek.bib:\\
\centerline{\textsf{asc.tuwien.ac.at/\~{}funkana/woracek/bibtex-database-woracek.bib}}
\\[2mm]
\textcircledP\ Preliminary version Wed 16 Dec 2015 23:04
\\[2mm]
\ding{229}\quad to obtain proper eepic-pictures use latex/dvipdf to compile
\\[-2mm]
}
}
\IfFileExists{bibtex-database-woracek.bib}
{}
{
\hspace*{0pt}\\[4mm]\centerline{\large\sf !!! File bibtex-database-woracek.bib does not exist !!!}
\\[-4mm]
}
\tableofcontents
\listoftodos
%
\newpage
\pagenumbering{arabic}
\setcounter{page}{1}
}
{}


%
%
%
\section{Introduction}

Let $ (s_n)_{n=0}^\infty$ be a of sequence real numbers, and assume that the Hamburger power moment problem for this sequence 
is solvable and indeterminate. Then the totality of all positive Borel measures on $\bb R$ with 
power moments $(s_n)_{n=0}^\infty$ is parameterised via 
their Cauchy-transforms with help of an entire $2\!\times\!2$-matrix function called the \emph{Nevanlinna matrix} of the moment sequence. 
A classical result of M.Riesz is that the entries of this matrix are entire functions of minimal exponential type, cf.\ \cite{riesz:1923a}. 
Further, all the entries of the Nevanlinna matrix have the same order, cf.\ \cite{baranov.woracek:smsub,berg.pedersen:1994}. 
We denote this common number by $\rho((s_n)_{n=0}^\infty)$ and call it the \emph{order of the moment sequence}. 

The order can be computed for several moment sequences for which the moment problem is explicitly solvable. 
As for theorems rather than examples, the only known estimate for $\rho((s_n)_{n=0}^\infty)$ in terms of the sequence $(s_n)_{n=0}^\infty$ 
itself (and  -- probably -- the first result in this context dealing with growth properties other than the exponential type) is due to 
M.S.Liv\v{s}ic back in 1939, cf.\ \cite{livshits:1939}. It asserts that 
\begin{equation}\label{Livsic}
	\rho((s_n)_{n=0}^\infty)\geq\limsup_{n\to\infty} \frac{2n\ln n}{\ln s_{2n}}
	,
\end{equation}
the right hand side being the order of the entire function $\sum_{n=0}^\infty \frac{z^{2n}}{s_{2n}}$. 
The question whether there exist moment problems for which the order is different from its Liv\v{s}ic estimate appears to have remained 
open since then. In particular, it is mentioned as such in \cite{berg.szwarc:2014}. 
The difficulty can be explained as follows. 
Let $P_n(z)=\sum_{k=0}^n b_{k,n}z^k$, $n\in\bb N_0$, be the orthonormal polynomials of the first kind associated with the 
moment sequence $(s_n)_{n=0}^\infty$. 
Then the order $\rho((s_n)_{n=0}^\infty)$ 
is expressed in terms of the coefficients $b_{k,n}$ as 
\begin{equation}\label{R48}
	\rho((s_n)_{n=0}^\infty) = 
	\limsup_{k\to\infty} \frac{-2k\ln k}{\ln\sum\limits_{n=k}^\infty b_{k,n}^2}
	,
\end{equation}
cf.\ \cite[Theorem 3.1]{berg.szwarc:2014}. 
Liv\v{s}ic' estimate \eqref{Livsic} is obtained when dropping all summands but $b_{n,n}^2$. 
While the term $b_{n,n}$, being the leading coefficient of the orthonormal polynomial $P_n$, is easily expressed via the Jacobi parameters 
of the sequence $(s_n)_{n=0}^\infty$, see \eqref{R49} below, and can in turn be estimated by $s_{2n}$, 
the other terms $b_{k,n}$, $k<n$, are nearly impossible to control. 

In this paper we show that there exist Hamburger moment sequences whose order is strictly larger 
than the right hand side of \eqref{Livsic}. Actually, we shall provide examples of symmetric (meaning that $s_n=0$ for odd $n$) moment sequences for which the gap between the actual order and the Liv\v{s}ic estimate can be arbitrarily close to $1$, see \thref{Livcounter}. 

Another objective of  the present paper is to establish upper and lower bounds for the order of the indeterminate moment problem in terms of the corresponding canonical system. The Hamburger moment problems correspond to canonical systems whose Hamiltonian $H$ has a very particular form \cite{kac:1999}. Namely, $H$ has determinant zero, is 
piecewise constant, and constancy intervals accumulate only to its right endpoint, cf.\ \thref{M28}. 
We employ the recent work \cite{romanov:20XX} about the order of the monodromy matrix of a canonical system 
to establish an upper estimate in \thref{M2}. The lower estimate given in \thref{R2} is easy to see and follows, e.g., from 
\cite{berg.szwarc:2014}. Under a weak regularity assumption these bounds coincide, and hence yield a formula for order, cf.\ \thref{R24}.

Structuring of the paper is as follows. In the remaining part of this introduction we recall the connections between 
moment sequence and Jacobi matrices on the one hand, and canonical systems on the other. It is vital to have this connection 
on hand, since our proofs proceed via the theory of canonical systems. 
Section~2 is in some sense the core of the paper. We establish the upper and lower bounds for order and discuss regularily distributed 
sequences. 
The subject of Section~3 is the construction of examples showing that equality in Liv\v{s}ic' estimate \eqref{Livsic} may fail. 
In fact we show -- slightly stronger -- that the bound obtained from \eqref{R48} by dropping all summands but $b_{n,n}^2$ 
can differ from the order by any pregiven number (only taking into account that the order is always between $0$ and $1$), 
cf.\ \thref{Livcounter}. 

\subsection*{Moment sequences, Jacobi matrices, and canonical systems}

We establish our results on order taking the viewpoint of canonical systems. 
To translate to moment sequences and/or Jacobi parameters, it is necessary to have these connections on hand explicitly. 

The relation between moment sequences and Jacobi matrices is most classical and commonly exploited. A standard reference 
is \cite{akhiezer:1961}. 
Given a moment sequence $(s_n)_{n=0}^\infty$, the orthogonal polynomials $P_n$, $n\in\bb N_0$, satisfy a three-term recurrence 
relation 
\begin{equation}\label{M51}
	zP_n(z)=\rho_nP_{n+1}(z)+q_nP_n(z)+\rho_{n-1}P_{n-1}(z),\quad n=0,1,2,\ldots
	.
\end{equation}
The coefficients $\rho_n$ and $q_n$ in this recurrency are called the \emph{Jacobi parameters} associated with 
the sequence $(s_n)_{n=0}^\infty$. They can be computed from the sequence $(s_n)_{n=0}^\infty$ via determinantal formulae. 
However, these expressions are hardly accessible to practical computation. 
One connection we need frequently is that 
\begin{equation}\label{R49}
	b_{n,n}=\bigg(\prod_{k=1}^{n-1}\rho_k\bigg)^{-1}
	.
\end{equation}
Let us now recall the -- maybe less commonly used -- relation with canoncial systems. The basic idea is that the 
three-term recurrence is nothing but a canonical system with a piecewise constant Hamiltonian. 
This idea can be made precise: Hamburger moment problems correspond to a certain type of canonical systems, and, under 
a suitable normalisation, this correspondence is one-to-one. 
An explicit presentation of these matters can be found in \cite{kac:1999}. 

Let $L\in(0,\infty]$ and $H \colon  [ 0 , L ) \to \bb R^{ 2 \times 2 } $ be a measurable function such that for almost every 
$x\in[0,L)$ the matrix $H(x)$ is positive semidefinite with $\tr H(x)=1$. 
Then the equation 
\[
	\frac{\partial}{\partial x}y(x,z)=-zJH(x)y(x,z),\quad x\in[0,L)
	,
\]
where $J:={\scriptsize\begin{pmatrix} 0 &\hspace*{-5pt} -1\\ 1 &\hspace*{-5pt} 0\end{pmatrix}}$ and $z$ is a complex 
parameter, is called the \emph{canonical system} with \emph{Hamiltonian} $H$. 
If $L<\infty$, we say that \emph{limit circle case} takes place for $H$, whereas for $L=\infty$ one speaks of 
\emph{limit point case}. 

\begin{definition}\thlab{M28}
	Let $\vec l=(l_n)_{n=1}^\infty$ be a sequence of positive numbers and let 
	$\vec\phi=(\phi_n)_{n=1}^\infty$ be a sequence of real numbers with $\phi_{n+1}\not\equiv\phi_n\mod\pi$, $n\in\bb N$. 
	Set 
	\begin{equation}\label{M32}
		x_0:=0,\qquad x_n:=\sum_{k=1}^nl_k,\ n\in\bb N,\qquad x_\infty:=\sum_{k=1}^\infty l_k\in(0,\infty]
		.
	\end{equation}
	Denote 
	\[
		\xi_\phi:=\binom{\cos\phi}{\sin\phi},\quad \phi\in\bb R
		.
	\]
	Then we call the function $H:[0,x_\infty)\to\bb R^{2\!\times\!2}$ defined by 
	\[
		H(x):=\xi_{\phi_n}\xi_{\phi_n}^*,\quad x\in[x_{n-1},x_n),\ n\in\bb N
		,
	\]
	the \emph{Hamburger Hamiltonian} with \emph{lengths} $\vec l$ and \emph{angles} $\vec\phi$. 
	Sometimes we refer to the points $x_n$ as the \emph{nodes} of $H$. 
	\begin{center}
	\setlength{\unitlength}{0.015mm}
	\begingroup\makeatletter\ifx\SetFigFont\undefined%
	\gdef\SetFigFont#1#2#3#4#5{%
	  \reset@font\fontsize{#1}{#2pt}%
	  \fontfamily{#3}\fontseries{#4}\fontshape{#5}%
	  \selectfont}%
	\fi\endgroup%
	{\renewcommand{\dashlinestretch}{30}
	\begin{picture}(6000,1000)(200,100)
	\thicklines
	\path(1320,500)(4470,500)
	\dottedline{120}(4470,500)(5820,500)
	\path(1320,400)(1320,600)
	\path(2220,400)(2220,600)
	\path(3120,400)(3120,600)
	\path(4020,400)(4020,600)
	\path(5820,400)(5820,600)
	\put(500,450){\makebox(0,0)[lb]{\smash{{\SetFigFont{10}{12.0}{\rmdefault}{\mddefault}{\updefault}$H$:}}}}
	\put(1320,40){\makebox(0,0)[b]{\smash{{\SetFigFont{8}{12.0}{\rmdefault}{\mddefault}{\updefault}$x_0$}}}}
	\put(2220,40){\makebox(0,0)[b]{\smash{{\SetFigFont{8}{12.0}{\rmdefault}{\mddefault}{\updefault}$x_1$}}}}
	\put(3120,40){\makebox(0,0)[b]{\smash{{\SetFigFont{8}{12.0}{\rmdefault}{\mddefault}{\updefault}$x_2$}}}}
	\put(4020,40){\makebox(0,0)[b]{\smash{{\SetFigFont{8}{12.0}{\rmdefault}{\mddefault}{\updefault}$x_3$}}}}
	\put(5820,40){\makebox(0,0)[b]{\smash{{\SetFigFont{8}{12.0}{\rmdefault}{\mddefault}{\updefault}$x_\infty$}}}}
	\put(1770,770){\makebox(0,0)[b]{\smash{{\SetFigFont{6}{12.0}{\rmdefault}{\mddefault}{\updefault}$\xi_{\phi_1}\xi_{\phi_1}^*$}}}}
	\put(2670,770){\makebox(0,0)[b]{\smash{{\SetFigFont{6}{12.0}{\rmdefault}{\mddefault}{\updefault}$\xi_{\phi_2}\xi_{\phi_2}^*$}}}}
	\put(3570,770){\makebox(0,0)[b]{\smash{{\SetFigFont{6}{12.0}{\rmdefault}{\mddefault}{\updefault}$\xi_{\phi_3}\xi_{\phi_3}^*$}}}}
	\put(1770,400){\makebox(0,0)[b]{\smash{{\SetFigFont{6}{8.0}{\rmdefault}{\mddefault}{\updefault}$\underbrace{\rule{35pt}{0pt}}_{l_1}$}}}}
	\put(2670,400){\makebox(0,0)[b]{\smash{{\SetFigFont{6}{8.0}{\rmdefault}{\mddefault}{\updefault}$\underbrace{\rule{35pt}{0pt}}_{l_2}$}}}}
	\put(3570,400){\makebox(0,0)[b]{\smash{{\SetFigFont{6}{8.0}{\rmdefault}{\mddefault}{\updefault}$\underbrace{\rule{35pt}{0pt}}_{l_3}$}}}}
	\end{picture}
	}
	\end{center}
\end{definition} 

\noindent
The correspondence between moment sequences, Jacobi parameters, and Hamburger Hamiltonians is given via the formulae 
(here $Q_n$, $n\in\bb N$, denote the orthogonal polynomials of the second kind) 
\begin{align}
	& l_n=P_n(0)^2+Q_n(0)^2,\quad n\in\bb N,
	\label{R54}
	\\
	& \frac 1{\rho_n} =|\sin(\phi_{n+1}-\phi_n)|\sqrt{l_nl_{n+1}}, \quad n\in\bb N,
	\label{R55}
	\\
	& q_n=-\frac 1{l_n}\big[\cot(\phi_{n+1}-\phi_n)+\cot(\phi_n-\phi_{n-1})\big], \quad n\in\bb N.
	\label{R56}
\end{align}
Of course, this relation is again somewhat implicit since the above formulae contain the off-diagonal Jacobi parameters and, 
by their structure, cannot easily be inverted. 
However, the moment sequence $(s_n)_{n=0}^\infty$ is indeterminate if and only if the sequence $\vec l=(l_n)_{n=1}^\infty$ is summable. 
And if $(s_n)_{n=0}^\infty$ is indeterminate, then its Nevanlinna matrix coincides with the monodromy matrix of the canonical system with 
the corresponding Hamburger Hamiltonian, i.e., the matrix $W(L,z)$ where $W(x,z)$ is the unique solution of the inital value problem 
\[
	\left\{
	\begin{array}{l}
		J\frac{\partial}{\partial x}W(x,z)=zH(x)W(x,z),\quad x\in[0,L],
		\\[2mm]
		W(0,z)=I.
	\end{array}
	\right.
\]

\section{Estimates for the order}

We use a pointwise and an averaged measure for the decay of a sequence of positive numbers. 

\begin{definition}\thlab{R46}
	Let $\vec y=(y_n)_{n=1}^\infty$ be a bounded sequence of positive real numbers and let $\alpha\geq 0$. 
	Then we set 
	\begin{align*}
		& \Delta^*(\vec y):=\sup\big\{\tau\geq 0:y_n={\rm O}(n^{-\tau})\big\},
		\\
		& \Delta(\vec y):=\sup\Big\{\tau\geq 0:\frac 1n\sum_{k=n}^{2n-1}y_k={\rm O}(n^{-\tau})\Big\},
		\\
		& \delta(\vec y,\alpha):=\liminf_{n\to\infty}G(n;\vec y,\alpha)
		\ \ \text{where}\ \ 
		G(n;\vec y,\alpha):=\frac{-1}{n\ln n}\ln\Big(y_n^\alpha\prod\limits_{k=1}^{n-1}y_k\Big).
	\end{align*}
\end{definition}

\noindent
The numbers $\Delta^*(\vec y)$ and $\Delta(\vec y)$ are understood as elements of $[0,\infty]$. 
Note here that the sets appearing in their definition are nonempty, since $\vec y$ is bounded. 
The number $\delta(\vec y,\alpha)$ is, a priori, an element of $[-\infty,\infty]$. 

\begin{lemma}\thlab{R47}
	Let $\vec y=(y_n)_{n=1}^\infty$ be a bounded sequence of positive real numbers and let $\alpha\geq 0$. 
	Then 
	\begin{equation}\label{R33}
		\Delta^*(\vec y)\leq\Delta(\vec y)\leq\delta(\vec y,\alpha)
		.
	\end{equation}
\end{lemma}
\begin{proof}
	The inequality $\Delta^*(\vec y)\leq\Delta(\vec y)$ is clear. 
	To show that $\Delta(\vec y)\leq\delta(\vec y,\alpha)$, 
	consider $\tau\geq 0$ and $c\geq 1$ with $\frac 1n\sum_{k=n}^{2n-1}y_k\leq cn^{-\tau}$, $n\in\bb N$. 

	For $m\in\bb N$, $m\geq 2$, let $p(m)  = \lfloor\log_2(m-1)\rfloor$ and $r(m)\in\{0,\ldots,2^{p(m)}-1\}$ be defined by $m-1=2^{p(m)}+r(m)$. The arithmetic-geometric mean inequality gives
	\begin{multline*}
		\prod_{k=1}^{m-1}y_k= \prod_{j=1}^{p(m)}\prod_{k=2^{j-1}}^{2^j-1}y_k
		\cdot\prod_{k=2^{p(m)}}^{m-1}y_k
		\\
		\leq \prod_{j=1}^{p(m)}
		\bigg(\underbrace{\frac 1{2^{j-1}}\sum_{k=2^{j-1}}^{2^j-1}y_k}_{\leq c(2^{j-1})^{-\tau}}\bigg)^{2^{j-1}}
		\cdot
		\bigg(\underbrace{\frac 1{r(m)+1}\sum_{k=2^{p(m)}}^{m-1}y_k}_{\leq\frac{2^{p(m)}}{r(m)+1} c(2^{p(m)})^{-\tau}}\bigg)^{r(m)+1}
		,
	\end{multline*}
	and we obtain
	\begin{align*}
		\frac{-1}{m\log_2 m} & \log_2\Big(y_m^\alpha\prod_{k=1}^{m-1}y_k\Big)\geq
		\frac 1{m\log_2 m}\Big[-\alpha\underbrace{\log_2 y_m}_{\leq\log_2 \left\| \vec y \right\|_\infty  }-
		\log_2 c \sum_{j=1}^{p(m)}2^{j-1}
		\\
		& +\tau \hspace*{-3mm}\underbrace{\sum_{j=1}^{p(m)}(j-1)2^{j-1}}_{=p(m)2^{p(m)} + O ( 2^{ p ( m )} ) }\hspace*{-3mm}
		-(r(m)\!+\!1)\log_2\frac{2^{p(m)}c}{r(m)\!+\!1}+(r(m)\!+\!1)\tau p(m) \Big]
		\\
		\geq &\ \tau\frac{p(m)}{\log_2 m}-
		\frac{r(m)\!+\!1}m\frac 1{\log_2 m}\log_2\frac{2^{p(m)}c}{r(m)\!+\!1}+\,{\rm o}(1)
		.
	\end{align*}
	Estimating the logarithm by its argument, we find that the second summand in the right hand side is ${\rm o}(1)$. Since  
	$\lim_{m\to\infty}\frac{p(m)}{\log_2 m}=1$ we infer that $\delta(\vec y,\alpha)\geq\tau$. 
\end{proof}

\subsection{An upper bound for $\rho(H)$}

For a Hamiltonian $H$ in the limit circle case we denote by $\rho(H)$ the order of the entries of its monodromy matrix. 
The following fact is, up to using the connection \eqref{R54}--\eqref{R56}, nothing but \cite[Theorem~1.2]{berg.szwarc:2014}. 

\begin{proposition}\thlab{R27}
	Let $H$ be a limit circle Hamburger Hamiltonian with lengths $\vec l$ and angles $\vec\phi$. 
	Then the order of $H$ does not exceed the convergence exponent of $\vec l$, i.e., 
	\[ 
		\rho(H)\leq\inf\big\{p>0:\vec l\in\ell^p\big\} 
		.
	\]
\end{proposition}

\noindent
In our first main result, \thref{M2} below, we give an upper bound for $\rho(H)$ which takes the asymptotic behaviour of 
the sequence of angles into consideration. 

To quantify the behaviour of length- and angle sequences of a Hamburger Hamiltonian, we use the power scale and a pointwise measure for the decay of lengths, 
an averaged measure for the decay of angle-differences, and a measure for the speed of possible convergence of angles
weighted with lengths, i.e., taking into account peaks of lengths. 

\begin{definition}\thlab{M1}
	Let $H$ be a Hamburger Hamiltonian with lengths $\vec l$ and angles $\vec\phi$. 
	Set
	\begin{align*}
		\Delta_l(H) &\, :=\Delta^*(\vec l),\qquad
		\Delta_l^+(H):=\max\big\{1,\Delta_l(H)\big\}
		,
		\\
		\Delta_\phi(H) &\, :=\Delta\big((|\sin(\phi_{n+1}-\phi_n)|)_{n=1}^\infty\big)
		.
	\end{align*} 
	Provided that $\Delta_l^+(H)<\infty$, set 
	\begin{align*}
		\Lambda(H) &\, :=\sup_{\phi\in[0,\pi)}\sup\Big\{\tau\geq 0:\sum_{j=n}^\infty l_j |\sin(\phi_j-\phi )|=
		{\rm O}\big(n^{1-\Delta_l^+-\tau}\big)\Big\}\in[0,\infty]
		.
	\end{align*} 
	When no confusion is possible, we drop explicit notation of $H$. 
\end{definition}

\noindent

\begin{remark}\thlab{M88}
	A pointwise estimate for the speed of convergence of angles has an implication on $\Lambda(H)$. Denoting  
	\begin{equation}\label{M83}
		\Lambda^*(H):=
		\begin{cases}
			\Delta^*\big((|\sin(\phi_n-\phi)|)_{n=1}^\infty\big) &\hspace*{-3mm},\quad 
			\phi\mod\pi=\lim_{n\to\infty}\phi_n\mod\pi\text{ exists},
			\\
			0 &\hspace*{-3mm},\quad \vec\phi\text{ not convergent modulo $\pi$},
		\end{cases}
	\end{equation}
	it holds that $\Lambda^*(H)\geq\Lambda(H)$. 
\end{remark}

\begin{lemma}\thlab{M80}
	The quantities $\Delta_\phi$ and $\Lambda$ are related by 
	\begin{equation}\label{R8}
		\Delta_\phi-1\leq\Lambda
		.
	\end{equation}
	In particular, if $\Delta_\phi<\infty$, then $\Delta_l^+-\Delta_\phi+\Lambda>0$ unless $\Delta_l^+=1$ and $\Delta_\phi=\Lambda+1$. 
\end{lemma}

\noindent
In this, and subsequent proofs, we use the following notation:
Let $X,Y$ be functions taking nonnegative numbers as values. Then 
\begin{align*}
	& X\lesssim Y
	\quad\Longleftrightarrow\quad
	\exists c>0;\ X\leq cY
	\\
	& X\asymp Y
	\quad\Longleftrightarrow\quad
	X\lesssim Y\ \text{and}\ Y\lesssim X
\end{align*}

\begin{proof}
	The inequality is trivial if $\Delta_\phi \leq 1$. Hence, assume that $\Delta_\phi>1$. 
	For arbitrary $\tau \in (1, \Delta_\phi)$ we have
	\[
		\sum_{j=n}^{\infty} |\sin( \phi_{j+1}-\phi_j ) |  \leq
		\sum_{l=0}^{\infty} \sum_{j=2^l n}^{2^{l+1}n} |\sin( \phi_{j+1}-\phi_j ) |
		\lesssim\sum_{l=0}^{\infty} (2^l n)^{1-\tau} = 
		{\rm O}\big(n^{1-\tau}\big)
		.
	\]
	By adding a proper multiple of $\pi$ to each $\phi_n$, we can assume without loss of generality that 
	$|\phi_{n+1}-\phi_n|\leq \frac{\pi}{2}$. Then 
	\[ 
		\sum_{j=n}^\infty|\phi_{j+1}-\phi_j| \lesssim
		\sum_{j=n}^\infty | \sin(\phi_{j+1}-\phi_j)|
		= {\rm O}\big(n^{1-\tau}\big)
		,
	\]
	hence $\vec\phi$ has the limit $\phi:=\phi_1+\sum_{n=1}^\infty(\phi_{n+1}-\phi_n)$. 
	If $\Delta_l<1$, and hence $\Delta_l^+=1$, we get
	\[
		\sum_{j=n}^\infty l_j |\sin(\phi_j-\phi )| \lesssim
		n^{1-\tau} \sum_{j=n}^\infty l_j = {\rm O}\big(n^{-(\tau-1)}\big)
		.
	\]
	If $\Delta_l\geq 1$, and hence $\Delta_l^+=\Delta_l$, we have for arbitrary $\epsilon>0$
	\[
		\sum_{j=n}^\infty l_j |\sin(\phi_j-\phi )| \lesssim
		\sum_{j=n}^\infty j^{-\Delta_l^+ + \epsilon}\cdot j^{1-\tau} = {\rm O}\big(n^{1 - \Delta_l^+ -(\tau-1) + \epsilon}\big)
		.
	\]
	In both cases, this shows that $\tau-1\leq \Lambda$.
\end{proof}

\noindent
Our first theorem can now be formulated. 

\begin{theorem}\thlab{M2}
	Let $H$ be a limit circle Hamburger Hamiltonian with lengths $\vec l$ and angles $\vec\phi$. 
	Assume that $(\Delta_l^+,\Delta_\phi,\Lambda)\neq(1,1,0)$. 
	\begin{enumerate}[$(i)$]
	\item Generic region: If $\Delta^+_l+\Delta_\phi\geq 2$, then 
	\[
		\rho(H)\leq \frac 1{\Delta^+_l+\Delta_\phi}
		.
	\]
	\item Critical triangle: If $\Delta^+_l+\Delta_\phi<2$, then 
		\[
			\rho(H)\leq\max\Big\{\frac 1{\Delta^+_l+\Delta_\phi},
			\frac{1-\Delta_\phi+\frac 12\Lambda}{\Delta^+_l-\Delta_\phi+\Lambda}\Big\}
			.
		\]
	\end{enumerate}
\end{theorem}

\noindent
The proof of this theorem will be carried out in \S2.2. Before that we discuss several aspects. 

\begin{remark}[Sharpness]\thlab{M4}
	We will see that for a large class of Hamburger Hamiltonians in the generic region 
	the stated upper bound is equal to their order, cf.\ \thref{R24}. 
	Contrasting this, in the critical triangle (and $\Delta_\phi>0$), we do not know whether the given bound is sharp. In fact, we have no example which lies inside the critical triangle 
	where $\rho(H)$ can be computed. In particular, it is unknown whether the speed of possible convergence of angles 
	measured by $\Lambda$ influences the order. We believe the answer is affirmative. 

	In the case excluded in the assumption, namely if $(\Delta_l^+,\Delta_\phi,\Lambda)=(1,1,0)$, 
	we have only the trivial bound ``$\rho(H)\leq 1$'', 
	and again do not know if this bound is attained by some Hamburger Hamiltonian. 
\end{remark}

\noindent
Note that, in some vague sense, the division into generic region and critical triangle corresponds to the cases of ``order $\leq\frac 12$'' or 
``order $\in(\frac 12,1]$'', respectively. That may be an explanation for the occurrence of this case distinction. As experience tells, 
it would not be a surprise to witness a fundamentally different behaviour in these cases. 

Next let us have a closer look at the two expressions whose maximum establishes the upper bound in the critical triangle. 
Set 
\begin{align*}
	& D:=\big\{(x,y,z)\in\bb R^3: x\geq 1,\,y\geq 0,\,z\geq 0,\,y\leq z+1,\,x-y+z>0\big\},
	\\
	& g(x,y,z):=\frac{1-y+\frac 12z}{x-y+z},\quad 
	(x,y,z)\in D
	.
\end{align*}
Then 
\begin{align}
	& \frac 1{x+y}\leq g(x,y,z)\ \Leftrightarrow\ 0\leq(2-x-y)(y-\frac 12z)
	\label{M84}
	\\
	& g(x,y,z)=\frac 12\ \text{if}\ x+y=2,\qquad g(x,y,z)=1\Leftrightarrow(x,z)=(1,0),
	\label{M87}
\end{align}
the function $g ( x,y, \cdot ) $ is decreasing if $ x+y < 2 $, increasing if $ x+y > 2 $, and the function 
$ g ( x , \cdot , z ) $ is monotone nonincreasing.  
	
The relation \eqref{M84} shows that the given bound in the critical triangle equals $\frac 1{\Delta_l^++\Delta_\phi}$ 
if and only if $\Delta_\phi\leq\frac 12\Lambda$. 

As we have already observed in \thref{R47} and \thref{M88}, pointwise estimates lead to estimates for $\Delta_\phi$ and $\Lambda$. 
From this we obtain the following corollary where the -- easier to handle -- quantities 
\[
	\Delta_\phi^*:=\Delta\big((|\sin(\phi_{n+1}-\phi_n)|)_{n=1}^\infty\big)
	,
\]
and $\Lambda^*$ appear instead of $\Delta_\phi$ and $\Lambda$. Of course, this statement is weaker than \thref{M2}. 

\begin{corollary}\thlab{M81}
	Let $H$ be a Hamburger Hamiltonian with lengths $\vec l$ and angles $\vec\phi$. 
	Assume that $(\Delta_l^+,\Delta_\phi^*,\Lambda^*)\neq(1,1,0)$. 
	Then 
	\[
		\rho(H)\leq 
		\begin{cases}
			\frac 1{\Delta^+_l+\Delta_\phi^*} &\hspace*{-3mm},\quad \Delta^+_l+\Delta_\phi^*\geq 2,
			\\
			\frac{1-\Delta_\phi^*+\frac 12\Lambda^*}{\Delta^+_l-\Delta_\phi^*+\Lambda^*} &\hspace*{-3mm},\quad 
			\Delta^+_l+\Delta_\phi^*<2.
		\end{cases}
	\]
\end{corollary}
\begin{proof}[Deduction of \thref{M81} from \thref{M2}]
	We distinguish three cases.
	\begin{list}{}{\leftmargin=0pt}
	\item --- \textit{Case $\Delta_l^++\Delta_\phi^*\geq 2$:} 
		We have $\Delta_l^++\Delta_\phi\geq\Delta_l^++\Delta_\phi^*\geq 2$. 
		Towards a contradiction assume that $(\Delta_l^+,\Delta_\phi,\Lambda)=(1,1,0)$. Then $\Delta_\phi^*\leq 1$ 
		and $\Delta_l^++\Delta_\phi^*\geq 2$ implies $\Delta_\phi^*=1$. Moreover, $\Lambda^*\leq\Lambda$, and 
		hence $\Lambda^*=0$. This case, however, is excluded by assumption. Thus \thref{M2} applies and yields 
		\[
			\rho(H)\leq\frac 1{\Delta_l^++\Delta_\phi}\leq\frac 1{\Delta_l^++\Delta_\phi^*}
			.
		\]
	\item --- \textit{Case $\Delta_l^++\Delta_\phi^*<2,\Delta_l^++\Delta_\phi\geq 2$:} 
		If $(\Delta_l^+,\Lambda)=(1,0)$, then also $\Lambda^*=0$, and hence $g(\Delta_l^+,\Delta_\phi^*,\Lambda^*)=1$, 
		cf.\ \eqref{M87}. 
		Assume that $(\Delta_l^+,\Lambda)\neq(1,0)$. Then \thref{M2} and the obvious fact that 
		$\Lambda^*\leq\Delta_\phi^*$ yields 
		\[
			\rho(H)\leq \frac 1{\Delta_l^++\Delta_\phi}\leq \frac 1{\Delta_l^++\Delta_\phi^*}
			\leq g(\Delta_l^+,\Delta_\phi^*,\Lambda^*)
			.
		\]
	\item --- \textit{Case $\Delta_l^++\Delta_\phi^*<2,\Delta_l^++\Delta_\phi<2$:} 
		We have 
		\begin{multline*}
			\rho(H)\leq\max\big\{\frac 1{\Delta_l^++\Delta_\phi},g(\Delta_l^+,\Delta_\phi,\Lambda)\big\}
			\\
			\leq\max\big\{\frac 1{\Delta_l^++\Delta_\phi^*},g(\Delta_l^+,\Delta_\phi^*,\Lambda^*)\big\}=
			g(\Delta_l^+,\Delta_\phi^*,\Lambda^*)
			.
		\end{multline*}
	\end{list}
\end{proof}

\noindent
Comparing the nature of the quantities $\Delta_\phi,\Lambda$ and $\Delta_\phi^*,\Lambda^*$, 
and having in mind \thref{R27}, suggests that there might be room for 
improvement by introducing an averaged measure for the decay of lengths. 
However, as examples show, there seems to be an intrinsic obstacle. 

\subsection{Proof of \thref{M2}}

To start with we settle two simple cases. 
\begin{list}{}{\leftmargin=0pt}
\item --- If $(\Delta_l^+,\Lambda)=(1,0)$ and hence $\Delta_\phi\leq 1$, or if $(\Delta_l^+,\Delta_\phi)=(1,0)$, 
	the assertion reduces to the trivial bound ``$\rho(H)\leq 1$''. 
\item --- The convergence exponent of $(l_n)_{n=0}^\infty$ is not larger than $\frac 1{\Delta_l^+}$. 
	Therefore \thref{R27} entails 
	\begin{equation}\label{R28}
		\rho(H)\leq \frac{1}{\Delta_l^+}
		.
	\end{equation}
	In particular, the assertion of the theorem holds if $\Delta_l^+=\infty$. 
\end{list}
These observations justify that throughout the following we may assume 
\[
	(\Delta_l^+,\Lambda)\neq(1,0),\ (\Delta_l^+,\Delta_\phi)\neq(1,0),\ \Delta_l<\infty
	.
\]
Note that these assumptions imply that the right hand side in asserted bound is strictly less than $1$. 

We are going to employ \cite[Theorem~1]{romanov:20XX} which provides an upper bound for the order 
of a (arbitrary) Hamiltonian. This theorem is based on finding appropriate approximations of a given Hamiltonian by simple ones. 

\begin{definition}\thlab{M16}
	Let $N\in\bb N$, let $(l_n)_{n=1}^N$ be a finite sequence of positive numbers, and let 
	$(\phi_n)_{n=1}^N$ be a finite sequence of real numbers with $\phi_{n+1} \not \equiv \phi_n \mod \pi$, $n=1,\ldots, N-1$.
	Set 
	\[
		x_0:=0,\qquad x_n:=\sum_{k=1}^nl_n,\ n=1,\ldots,N
		.
	\]
	Then we speak of the function $H:[0,x_N)\to\bb R^{2\!\times\!2}$ which is defined by 
	\[
		H(x):=\xi_{\phi_n}\xi_{\phi_n}^*,\quad x\in[x_{n-1},x_n),\ n=1,\ldots,N
		,
	\]
	as the \emph{finite rank Hamiltonian with parameters} $\langle N,(l_n)_{n=1}^N,(\phi_n)_{n=1}^N\rangle$. 
	\begin{center}
	\setlength{\unitlength}{0.015mm}
	\begingroup\makeatletter\ifx\SetFigFont\undefined%
	\gdef\SetFigFont#1#2#3#4#5{%
	  \reset@font\fontsize{#1}{#2pt}%
	  \fontfamily{#3}\fontseries{#4}\fontshape{#5}%
	  \selectfont}%
	\fi\endgroup%
	{\renewcommand{\dashlinestretch}{30}
	\begin{picture}(6000,1000)(200,100)
	\thicklines
	\path(1320,500)(3520,500)
	\dottedline{120}(3520,500)(4520,500)
	\path(4520,500)(5820,500)
	\path(1320,400)(1320,600)
	\path(2220,400)(2220,600)
	\path(3120,400)(3120,600)
	\path(4920,400)(4920,600)
	\path(5820,400)(5820,600)
	\put(500,450){\makebox(0,0)[lb]{\smash{{\SetFigFont{10}{12.0}{\rmdefault}{\mddefault}{\updefault}$H$:}}}}
	\put(1320,40){\makebox(0,0)[b]{\smash{{\SetFigFont{8}{12.0}{\rmdefault}{\mddefault}{\updefault}$x_0$}}}}
	\put(2220,40){\makebox(0,0)[b]{\smash{{\SetFigFont{8}{12.0}{\rmdefault}{\mddefault}{\updefault}$x_1$}}}}
	\put(3120,40){\makebox(0,0)[b]{\smash{{\SetFigFont{8}{12.0}{\rmdefault}{\mddefault}{\updefault}$x_2$}}}}
	\put(4920,40){\makebox(0,0)[b]{\smash{{\SetFigFont{8}{12.0}{\rmdefault}{\mddefault}{\updefault}$x_{N-1}$}}}}
	\put(5820,40){\makebox(0,0)[b]{\smash{{\SetFigFont{8}{12.0}{\rmdefault}{\mddefault}{\updefault}$x_N$}}}}
	\put(1770,770){\makebox(0,0)[b]{\smash{{\SetFigFont{6}{12.0}{\rmdefault}{\mddefault}{\updefault}$\xi_{\phi_1}\xi_{\phi_1}^*$}}}}
	\put(2670,770){\makebox(0,0)[b]{\smash{{\SetFigFont{6}{12.0}{\rmdefault}{\mddefault}{\updefault}$\xi_{\phi_2}\xi_{\phi_2}^*$}}}}
	\put(5370,770){\makebox(0,0)[b]{\smash{{\SetFigFont{6}{12.0}{\rmdefault}{\mddefault}{\updefault}$\xi_{\phi_N}\xi_{\phi_N}^*$}}}}
	\put(1770,400){\makebox(0,0)[b]{\smash{{\SetFigFont{6}{8.0}{\rmdefault}{\mddefault}{\updefault}$\underbrace{\rule{35pt}{0pt}}_{l_1}$}}}}
	\put(2670,400){\makebox(0,0)[b]{\smash{{\SetFigFont{6}{8.0}{\rmdefault}{\mddefault}{\updefault}$\underbrace{\rule{35pt}{0pt}}_{l_2}$}}}}
	\put(5370,400){\makebox(0,0)[b]{\smash{{\SetFigFont{6}{8.0}{\rmdefault}{\mddefault}{\updefault}$\underbrace{\rule{35pt}{0pt}}_{l_N}$}}}}
	\end{picture}
	}
	\end{center}
\end{definition}

\begin{theorem}[\cite{romanov:20XX}]\thlab{M29}
	Let $L\in(0,\infty)$, and let $H:[0,L)\to\bb R^{2\!\times\!2}$ be a Hamiltonian with $\tr H=1$ a.e.
	Let $d\in(0,1]$, and assume that there exists a family of finite rank Hamiltonians 
	\[
		H^{\star}(R),\ R>1\quad \Big(\text{parameters
		$\big\langle N^{\star}(R),(l_n^{\star}(R))_{n=1}^{N^{\star}(R)},(\phi_n^{\star}(R))_{n=1}^{N^{\star}(R)}\big\rangle$}\Big)
	\]
	and a family of sequences of weights 
	\[
		\big(a_n(R)\big)_{n=1}^{N^{\star}(R)},\ R>1\quad\text{with}\quad a_n(R)\in(0,1]
		,
	\]
	such that (\/${\rm O}$-notation is understood for $R\to\infty$ and $\|.\|$ denotes any matrix norm)
	\begin{enumerate}[$(i)$]
	\item ${\displaystyle \sum_{k=1}^{N^{\star}(R)}\frac 1{a_k(R)^2}\int\limits_{x_{k-1}^{\star}(R)}^{x_k^{\star}(R)}
		\big\|H(x)-[H^{\star}(R)](x)\big\|\,dx
		={\rm O}\big(R^{d-1}\big)}$,
	\item ${\displaystyle \sum_{k=1}^{N^{\star}(R)}a_k(R)^2l_k^{\star}(R)
		={\rm O}\big(R^{d-1}\big)}$,
	\item ${\displaystyle \sum_{k=1}^{N^{\star}(R)-1}
		\ln\bigg(1+\frac{\big|\sin\big(\phi_{k+1}^{\star}(R)-\phi_k^{\star}(R)\big)\big|}{a_{k+1}(R)a_k(R)}\bigg)
		={\rm O}\big(R^d\big)}$,
	\item ${\displaystyle \big|\ln a_{1}(R)\big|+\big|\ln a_{N^{\star}(R)}(R)\big|+
		\sum_{k=1}^{N^{\star}(R)-1}\Big|\ln\frac{a_{k+1}(R)}{a_k(R)}\Big|
		={\rm O}\big(R^d\big)}$.
	\end{enumerate}
	Then the order of the entries of the monodromy matrix of $H$ does not exceed $d$. 
\end{theorem}

\begin{ndefinition}{Notation}\thlab{M74}
	It turns out useful to agree on the following abbreviations.
	\begin{enumerate}[$(i)$]
	\item If $X,Y:[1,\infty)\to[0,\infty)$, then we define 
		\[
			X(R)\preceq Y(R)
			\quad \Longleftrightarrow\quad
			\forall\,\epsilon>0\ \exists\,C>0\ \forall\,R\geq 1:\ X(R)\leq CR^\epsilon Y(R)
		\]
	\item If $x,y\in\bb R$ with $x<y$, and $X_k\in\bb C$ for $k\in[x,y]\cap\bb Z$, then we write 
		\[
			\sum_{k\geq x}^yX_k:=\sum_{k\in[x,y]\cap\bb Z}X_k
			.
		\]
	\end{enumerate}
\end{ndefinition}
	
\noindent
We set $\Delta_\phi':=\Delta_\phi$ if $\Delta_\phi<\infty$, and let $\Delta_\phi'$ be an arbitrary number larger than $1$ 
if $\Delta_\phi=\infty$. 
Next, for $\phi\in[0,\pi)$, set 
\[  
	\Lambda(\phi):=\sup\Big\{\tau\geq 0:\sum_{j=n}^\infty l_j|\sin(\phi_j-\phi)|={\rm O}\big(n^{1-\Delta_l^+-\tau}\big)\Big\}
	\in[0,\infty]
	.
\]
Again, set $\Lambda(\phi)':=\Lambda(\phi)$ if $\Lambda(\phi)<\infty$, and let 
$\Lambda(\phi)'$ be an arbitrary number larger than $1$ if $\Lambda(\phi)=\infty$. 

Consider $\phi\in[0,\pi)$ such that $(\Delta_l^+,\Lambda(\phi)')\neq (1,0)$, and let $d(\phi)$ be any number with 
\begin{equation}\label{M76}
	1>d(\phi)>
	\begin{cases}
		\frac 1{\Delta^+_l+\Delta_\phi'} 
		&\hspace*{-3mm},\quad \Delta^+_l+\Delta_\phi'\geq 2,
		\\[2mm]
		\max\Big\{\frac 1{\Delta^+_l+\Delta_\phi'},
		\frac{1-\Delta_\phi'+\frac 12\Lambda(\phi)'}{\Delta^+_l-\Delta_\phi'+\Lambda(\phi)'}\Big\}
		&\hspace*{-3mm},\quad \Delta^+_l+\Delta_\phi'<2.
	\end{cases}
\end{equation}
Given $R>1$ we define an approximating Hamiltonian as a cut-off of $H$ prolonged by one interval with angle $\phi$. 
The cutting point will be the node $x_{N(R)}$ where 
\[
	N(R)=\Big\lfloor R^{\frac{1-d(\phi)}{\Delta_l^+-1+\Lambda(\phi)'/2}}\Big\rfloor
	.
\]
Note that the value $(1-d(\phi))(\Delta_l^+-1+\Lambda(\phi)'/2)^{-1}$ appearing in the exponent is positive. 
Now we define 
\begin{align*}
	& N^{\star}(R):=N(R)+1,
	\\
	& l_n^{\star}(R):=
	\begin{cases}
		l_n &\hspace*{-3mm},\quad n=1,\ldots,N(R)
		\\
		x_\infty-x_{N(R)} &\hspace*{-3mm},\quad n=N^{\star}(R)
	\end{cases}
	\\
	& \phi_n^{\star}(R):=
	\begin{cases}
		\phi_n &\hspace*{-3mm},\quad n=1,\ldots,N(R)
		\\
		\phi &\hspace*{-3mm},\quad n=N^{\star}(R)
	\end{cases}
\end{align*}
and let $H^{\star}(R)$ be the finite rank Hamiltonian given by this data. 

The required weights $a_n(R)$ are defined by (here we set $\sigma:=(\Delta_l^++\Delta_\phi')^{-1}$)
\[
	a_n(R)^2 :=
	\begin{cases}
		\frac 1R n^{\Delta_l} &\hspace*{-3mm},\quad 1\leq n\leq R^{\sigma},
		\\
		\frac 1{\sqrt R} n^{\frac 12(\Delta_l^+-\Delta_\phi')} &\hspace*{-3mm},\quad R^{\sigma}<n\leq N(R),
		\\
		n^{-\frac 12\Lambda(\phi)'} &\hspace*{-3mm},\quad n=N^{\star}(R)
		.
	\end{cases}
\]
We need to check that $a_n(R)\leq 1$. This is clear in all cases except when $n\in(R^{\sigma},N(R)]$ and $\Delta_l^+>\Delta_\phi'$. 
Then it amounts to showing 
\begin{equation}\label{M75}
	\frac{1-d(\phi)}{\Delta_l^+-1+\Lambda(\phi)'/2}(\Delta_l^+-\Delta_\phi')\leq 1
	,
\end{equation}
or, equivalently,  
\begin{equation}\label{M755}
	\frac{1-\Delta_\phi'-\frac 12\Lambda(\phi)'}{\Delta_l^+-\Delta_\phi'}\leq d(\phi)
	.
\end{equation}
To this end, notice that
\[
		\frac{1-\Delta_\phi'-\frac 12\Lambda(\phi)'}{\Delta_l^+-\Delta_\phi'}\leq
		\frac{1-\Delta_\phi'+\frac 12\Lambda(\phi)'}{\Delta_l^+-\Delta_\phi'+\Lambda(\phi)'}
		,
	\]
for $ a/b \le ( a +x )/ ( b+x) $ if $ a \le b $, $ x \ge 0 $, $ b > 0 $.  This implies \eqref{M755} in the case       $\Delta_l^++\Delta_\phi'<2$. In the case  $\Delta_l^++\Delta_\phi'\geq 2$ the inequality
	\[
		\frac{1-\Delta_\phi'}{\Delta_l^+-\Delta_\phi'}\leq\frac 1{\Delta_l^++\Delta_\phi'} , 
	\]
	holds, and \eqref{M755} follows because $ \Lambda(\phi)' \ge 0 $.

We now show that with the above approximation and $d:=d(\phi)+\epsilon$, where $\epsilon>0$ is arbitrary, 
the hypotheses of \thref{M29} are satisfied. 
To shorten notation, we drop the argument $R$ whenever convenient.

\hspace*{0pt}\\[2mm]
\textit{Item $(i)$:}
\begin{multline*}
	\sum_{k=1}^{N+1}\frac 1{a_k^2} \int_{x^{\star}_{k-1}}^{x^{\star}_k}\|H(x)-H^{\star}(x)\|\,dx=
	\frac 1{a_{N+1}^2}\int_{x_{N}}^{x_\infty}\|H(x)-\xi_\phi \xi_\phi^T\|\,dx
	\\
	\lesssim N^{ \Lambda(\phi)'/2 } \sum_{j=N+1}^\infty l_j|\sin(\phi_j-\phi)|
	\preceq N^{ \Lambda(\phi)'/2 } N^{ 1 - \Delta^+_l - \Lambda(\phi)' }\leq R^{d(\phi)-1} 
	.
\end{multline*}

\hspace*{0pt}\\[2mm]
\textit{Item $(ii)$:}
\begin{equation}\label{M78}
	\sum_{k=1}^{N+1}a_k^2l_k^{\star} = \frac 1R\sum_{k\geq 1}^{R^{\sigma}}k^{\Delta_l}l_k
	+\frac 1{\sqrt R}\sum_{k\geq R^{\sigma}}^N k^{(\Delta_l^+-\Delta_\phi')/2}l_k+N^{-\Lambda(\phi)'/2}(x_\infty-x_N)
	.
\end{equation}
 Since $l_k\preceq k^{-\Delta_l}$, the first term in the right hand side satisfies 
\[
	\frac 1R\sum_{k\geq 1}^{R^{\sigma}}k^{\Delta_l}l_k\preceq R^{\sigma-1}\leq R^{d(\phi)-1} 
	.
\]
Since $x_\infty-x_N=\sum_{k=N+1}^\infty l_k\preceq N^{1-\Delta_l^+}$, we have 
\[
	N^{-\Lambda(\phi)'/2}(x_\infty-x_N)\preceq N^{-(\Delta_l^+-1+\Lambda(\phi)'/2)}\lesssim R^{d(\phi)-1}
	.
\]
In order to estimate the second term on the right side of \eqref{M78}, we distinguish the cases that $\Delta_l^+>1$ and $\Delta_l^+=1$. 
\begin{list}{}{\leftmargin=0pt}
\item --- \textit{Case $\Delta_l^+>1$:} 
	Then $\Delta_l=\Delta_l^+$, and we obtain 
	\begin{align*}
		\frac 1{\sqrt R}\sum_{k\geq R^{\sigma}}^N k^{(\Delta_l^+-\Delta_\phi')/2}l_k\preceq &\ 
		\frac 1{\sqrt R}\sum_{k\geq R^{\sigma}}^N k^{-(\Delta_l^++\Delta_\phi')/2}
		\\
		\lesssim &\ \frac 1{\sqrt R}
		\begin{cases}
			(R^{\sigma})^{1-(\Delta_l^++\Delta_\phi')/2} &\hspace*{-3mm},\quad \Delta_l^++\Delta_\phi'>2
			\\
			\ln N &\hspace*{-3mm},\quad \Delta_l^++\Delta_\phi'=2
			\\ 
			N^{1-(\Delta_l^++\Delta_\phi')/2} &\hspace*{-3mm},\quad \Delta_l^++\Delta_\phi'<2
		\end{cases}
	\end{align*}
	In the first case, since $\sigma\leq\frac 12$ and $(\Delta_l^++\Delta_\phi)/2\geq 1$, it holds that 
	$(R^{\sigma})^{1-(\Delta_l^++\Delta_\phi')/2}=R^{\sigma-\frac 12}\leq R^{d(\phi)-\frac 12}$. 
	In the second case $d(\phi)>\frac 12$ and hence $\ln N\lesssim R^{d(\phi)-\frac 12}$. In the third case, \eqref{M755} implies that
	\[
		\frac{1-d(\phi)}{\Delta_l^+-1+\Lambda(\phi)'/2}\Big(1-\frac{\Delta_l^++\Delta_\phi'}2\Big)\leq d(\phi)-\frac 12 .
	\]
We find that $N^{1-(\Delta_l^++\Delta_\phi')/2}\lesssim R^{d(\phi)-\frac 12}$. 
\item --- \textit{Case $\Delta_l^+=1$:}
	\begin{align*}
		\frac 1{\sqrt R}\sum_{k\geq R^{\sigma}}^N k^{(\Delta_l^+-\Delta_\phi')/2}l_k \leq &\ 
		\frac 1{\sqrt R}\Big[\sum_{k\geq R^{\sigma}}^N l_k\Big]\max_{R^{\sigma}\leq k\leq N}k^{(1-\Delta_\phi')/2}
		\\
		\leq &\ \frac 1{\sqrt R}
		\begin{cases}
			(R^{\sigma})^{(1-\Delta_\phi')/2} &\hspace*{-3mm},\quad \Delta_\phi'\geq 1
			\\
			N^{(1-\Delta_\phi')/2} &\hspace*{-3mm},\quad \Delta_\phi'<1
		\end{cases}
	\end{align*}
	In the case $ \Delta_\phi^\prime \ge 1 $ observe that 
	\[
		\sigma\frac{1-\Delta_\phi'}2\leq d(\phi)-\frac 12
		\ \Leftrightarrow\ 
		1-\Delta_\phi'\leq2\frac{d(\phi)}{\sigma}-(1+\Delta_\phi')
		\ \Leftrightarrow\ 
		1\leq\frac{d(\phi)}{\sigma}
		,
	\]
	and in the case $ \Delta_\phi^\prime < 1 $ that
	\[
		\frac{1-d(\phi)}{\Lambda (\phi)^\prime /2}\frac{1-\Delta_\phi'}2\leq d(\phi)-\frac 12
		\ \Leftrightarrow\ 
		1-\Delta_\phi'+\frac{\Lambda(\phi)'}2\leq d(\phi)\big[1-\Delta_\phi'+\Lambda(\phi)'\big]
		.
	\]
\end{list}

Thus, the right hand side in \eqref{M78} is ${\rm O}( R^{ d (\phi)-1 } ) $ in all cases.

\hspace*{0pt}\\[2mm]
\textit{Item $(iii)$:}
The weights $a_n$ are, independently of $n$, bounded from below by an appropriate power of $R$. Hence, 
we have $\ln\Big(1+\frac{|\sin(\phi_{k+1}^{\star}-\phi_k^{\star})|}{a_{k+1}a_k}\Big)\lesssim\ln R$, and obtain 
\begin{align*}
	\sum_{k\geq 1}^{R^{\sigma}+1}
	\ln\bigg(1+\frac{|\sin(\phi_{k+1}^{\star}-\phi_k^{\star})|}{a_{k+1}a_k}\bigg)
	&\ \lesssim R^{\sigma}\ln R\lesssim R^{d(\phi)}
	,
	\\
	\ln\bigg(1+\frac{|\sin(\phi_{N+1}^{\star}-\phi_N^{\star})|}{a_{N+1}a_N}\bigg)
	&\ \lesssim\ln R\lesssim R^{d(\phi)}
	.
\end{align*}
In the remaining part of the sum the second line of the definition of weights applies to the effect that
\begin{align*}
	\sum_{k\geq R^{\sigma}+1}^{N-1} &
	\ln\bigg(1+\frac{|\sin(\phi_{k+1}^{\star}-\phi_k^{\star})|}{a_{k+1}a_k}\bigg)
	\leq 
	\sum_{k\geq R^{\sigma}+1}^{N-1} \frac 1{a_{k+1}a_k}\big|\sin(\phi_{k+1}-\phi_k)\big|
	\\
	&\ =\sum_{k\geq R^{\sigma}+1}^{N-1} 
	\sqrt R\cdot[(k+1)k]^{-\frac 14(\Delta_l^+-\Delta_\phi')}\cdot\big|\sin(\phi_{k+1}-\phi_k)\big|
	\\
	&\ \lesssim
	\sqrt R\cdot\sum_{j\geq 1}^{\log_2(N/R^{\sigma})+1} 
	\Big(\sum_{k\geq 2^{j-1}R^{\sigma}}^{2^jR^{\sigma}-1}\big|\sin(\phi_{k+1}-\phi_k)\big|\,\Big)
	\max_{k\in[2^{j-1}R^{\sigma},2^jR^{\sigma}]}k^{-\frac 12(\Delta_l^+-\Delta_\phi')}
	\\
	&\ \lesssim 
	\sqrt R\cdot\sum_{j\geq 1}^{\log_2(N/R^{\sigma})+1} 
	\Big(\sum_{k\geq 2^{j-1}R^{\sigma}}^{2^jR^{\sigma}-1}\big|\sin(\phi_{k+1}-\phi_k)\big|\,\Big)
	\big(2^jR^{\sigma}\big)^{-\frac 12(\Delta_l^+-\Delta_\phi')} 
	\\
	&\ \preceq
	\sqrt R\cdot(R^{\sigma})^{-\frac 12(\Delta_l^+-\Delta_\phi')}
	\sum_{j\geq 1}^{\log_2(N/R^{\sigma})+1}
	\big(2^{j-1}R^{\sigma}\big)^{1-\Delta_\phi'}\cdot 2^{-j\frac 12(\Delta_l^+-\Delta_\phi')}
	\\
	&\ \asymp
	\sqrt R\cdot(R^{\sigma})^{1-(\Delta_l^++\Delta_\phi')/2}
	\sum_{j\geq 1}^{\log_2(N/R^{\sigma})+1} 2^{j[1-(\Delta_l^++\Delta_\phi')/2]}
	\\
	&\ \lesssim
	\sqrt R\cdot
	\begin{cases}
		(R^{\sigma})^{1-(\Delta_l^++\Delta_\phi')/2} &\hspace*{-3mm},\quad \Delta_l^++\Delta_\phi'>2
		\\
		\ln R &\hspace*{-3mm},\quad \Delta_l^++\Delta_\phi'=2
		\\ 
		N^{1-(\Delta_l^++\Delta_\phi')/2} &\hspace*{-3mm},\quad \Delta_l^++\Delta_\phi'<2
	\end{cases}
\end{align*}
By what we showed in the proof of ``Item~$(ii)$'', the last expression is in all cases $\lesssim R^{d(\phi)}$. 

\hspace*{0pt}\\[2mm]
\textit{Item $(iv)$:}
As we already observed, the weights $a_n$ are bounded below by some power of $R$. Hence, each summand appearing in 
``Item~$(iv)$'' is $\lesssim\ln R$. Moreover, we have 
\[
	\Big(\frac{a_{k+1}}{a_k}\Big)^2=
	\begin{cases}
		\big(\frac{k+1}k\big)^{\Delta_l} &\hspace*{-3mm},\quad 1\leq k\leq R^{\sigma}-1,
		\\
		\big(\frac{k+1}k\big)^{\frac 12(\Delta_l^+-\Delta_\phi')} &\hspace*{-3mm},\quad R^{\sigma}<k\leq N-1,
	\end{cases}
\]
and together it follows that 
\[
	\big|\ln a_1\big|+\big|\ln a_{N+1}\big|+\sum_{k=1}^N\Big|\ln\frac{a_{k+1}}{a_k}\Big|
	\lesssim\ln R+\sum_{k=1}^N\ln\big(1+\frac 1k\big)\lesssim\ln R\lesssim R^{d(\phi)}
	.
\]
We see that \thref{M29} is indeed applicable, and yields that $\rho(H)\leq d(\phi)+\epsilon$. 

The proof of \thref{M2} is completed 
by letting $\epsilon\searrow 0$, passing to the infimum over all $d(\phi)$ subject to \eqref{M76}, passing to the supremum over all 
$\phi$ subject to $(\Delta_l^+,\Lambda(\phi)')\neq (1,0)$, and -- if necessary -- letting $\Delta_\phi'\nearrow\infty$ and 
$\Lambda(\phi)'\nearrow\infty$.

\begin{remark}\thlab{M89}
	Choosing the approximating Hamiltonian as a cut-off of $H$ is of course natural. 
	The choice of the weights $ a_j $ in the proof is based on term-by-term optimization in conditions $(ii)$ and $(iii)$ 
	assuming that the replacement of $ a_j a_{ j+1 } $ in the denominator in $(iii)$ by $ a_j^2 $ does not spoil the estimate too much. 
\end{remark}

\subsection{A lower bound for $\rho(H)$}

The following limes inferior appears naturally in estimating order from below. 

\begin{definition}\thlab{R1}
	Let $H$ be a Hamburger Hamiltonian with lengths $\vec l$ and angles $\vec\phi$. 
	Then we set 
	\[
		\delta_{l,\phi}(H):=\liminf_{m\to\infty}\Big[
		G(m;\vec l,{\textstyle\frac 12})+G\big(m;(|\sin(\phi_{n+1}-\phi_n)|)_{n=1}^\infty,0\big)
		\Big]
		.
	\]
\end{definition}

\noindent
The meaning of $\delta_{l,\phi}$ is easily explained. 
Recall the notation  $b_{n,n}$ for the leading coefficient of the orthogonal polynomial of the first kind of degree $n$. 
Plugging \eqref{R49} into \eqref{R55} we have 
\[
	b_{m,m}=\bigg(\prod_{k=1}^{m-1}\rho_k\bigg)^{-1}=\prod_{k=1}^{m-1}|\sin(\phi_{k+1}-\phi_k)|\sqrt{l_kl_{k+1}},
\]
from whence 
\begin{equation}\label{R37}
	\delta_{l,\phi}=\liminf_{m\to\infty}\frac{-\ln b_{m,m}}{m\ln m}
\end{equation}
The following fact is now nothing but \cite[Proposition 7.1(iii)]{berg.szwarc:2014}. 

\begin{proposition}\thlab{R2}
	Let $H$ be a limit circle Hamburger Hamiltonian with lengths $\vec l$ and angles $\vec\phi$. 
	Then
	\begin{equation}\label{R51}
		\rho(H)\geq\frac 1{\delta_{l,\phi}(H)}
		.
	\end{equation}
\end{proposition}
\begin{proof}
	According to \cite[Proposition 7.1(iii)]{berg.szwarc:2014}, $\rho(H)$ is greater or equal to the order of the entire function 
	$\sum_{n=0}^\infty b_{n,n} z^n$. The assertion follows from the standard formula for the order of an entire function 
	in terms of its Taylor coefficients, see, e.g., \cite[Theorem 2.2.2]{boas:1954}. 
\end{proof}

\begin{remark}\thlab{M77} 
	It is possible to prove \thref{R2} directly, i.e., without addressing the correspondence between Hamburger Hamiltonians and orthogonal
	polynomials. To this end, one has to use the multiplicative representation of the monodromy
	matrix $W(L,z)$ by monodromy matrices corresponding to intervals $(x_{j-1},x_j)$, and take into account the fact that the matrix
	elements of $W(x_j,z)$ are polynomials in $z$ with real roots and so their moduli at $z=i\tau$, $\tau\in\mathbb R$, are
	estimated from below by the absolute value of the respective leading coefficients.

	The interested reader is referred to the preprint version \cite{pruckner.romanov.woracek:jaco-ASC} of this article, 
	where we included this alternative proof of \thref{R2}.
\end{remark}

\noindent
Sometimes it is useful to look at the lengths and angles of a Hamburger Hamiltonian separately.
In fact, this viewpoint is vital when comparing the lower bound \eqref{R51} for $\rho(H)$ with the upper bound established in \thref{M2}. 
Also, it helps when considering concrete examples. 

\begin{definition}\thlab{R3}
	Let $H$ be a Hamburger Hamiltonian with lengths $\vec l$ and angles $\vec\phi$. 
	Then we set 
	\begin{align*}
		& \delta_l(H):=\liminf_{m\to\infty} G(m;\vec l,{\textstyle\frac 12}),
		\\
		& \delta_\phi(H):=\liminf_{m\to\infty} G\big(m;(|\sin(\phi_{n+1}-\phi_n)|)_{n=1}^\infty,0\big)
		.
	\end{align*}
\end{definition}

\noindent
The next statement is an immediate corollary of \thref{R2}.

\begin{corollary}\thlab{R52}
	Let $H$ be a Hamburger Hamiltonian with lengths $\vec l$ and angles $\vec\phi$. 
	Assume that at least one of $\delta_l$ and $\delta_\phi$ exists as a limit. Then
	\[
		\rho(H) \geq \frac{1}{\delta_l+\delta_\phi}
		.
	\]
\end{corollary}

\noindent
Let us provide an example that the additional assumption in the corollary cannot be dropped. 

\begin{proposition}\thlab{R36}
	Let $\alpha>\beta>1$ and $\gamma>0$ with $\beta+\frac\gamma2<\alpha<\beta+2\gamma$ be given. 
	Consider the Hamburger Hamiltonian $H$ with lengths and angles given by 
	\[ 
		l_n:=
		\begin{cases} 
			n^{-\alpha} &\hspace*{-3mm},\quad 2^{2j}\leq n<2^{2j+1},j\in\bb N_0,
			\\[2mm]
			n^{-\beta} &\hspace*{-3mm},\quad 2^{2j-1}\leq n<2^{2j},j\in\bb N, 
		\end{cases} 
	\] 
	\[ 
		\phi_1:=0,\qquad
		\phi_{n+1}-\phi_n:=
		\begin{cases} 
			\frac \pi2 &\hspace*{-3mm},\quad 2^{2j}\leq n<2^{2j+1},j\in\bb N_0,
			\\[2mm]
			n^{-\gamma} &\hspace*{-3mm},\quad 2^{2j-1}\leq n<2^{2j},j\in\bb N. 
		\end{cases} 
	\] 
	Then $\rho(H)<\frac{1}{\delta_l+\delta_\phi}$. 
\end{proposition}
\begin{proof}
	By the choice of $\alpha$ being greater than $\beta$, 
	\[
		\delta_l(H)=\lim_{m\to\infty} G(2^{2m};\vec l,{\textstyle\frac 12})=\frac{2\beta+\alpha}3
		.
	\]
	Then 
	\[
		\delta_\phi(H)=\lim_{m\to\infty} G\big(2^{2m+1};(|\sin(\phi_{n+1}-\phi_n)|)_{n=1}^\infty,0\big)=\frac\gamma3
		.
	\] 
	On the other hand, if $M_n(z)$ is the monodromy matrix corresponding to $n$-th interval of the Hamiltonian, 
	then for any $\epsilon>0$
	\[ 
		\log\Big\|\prod_{n=2^{2j}}^{2^{2j+1}-1} M_n(z)\Big\|\leq\log\prod_{n=2^{2j}}^{2^{2j+1}}\Big(1+\frac{|z|}{n^\alpha}\Big)
		\leq C_\epsilon|z|^{1/\alpha+\epsilon}\sum_{n=2^{2j}}^{2^{2j+1}}\frac 1{n^{1+\epsilon}}
		,
	\]
	hence the product of the norms of $M_n$ over $n\in\bigcup_{j\in\bb N_0}[2^{2j},2^{2j+1})$ is estimated above by 
	$e^{C|z|^{1/\alpha+\epsilon}}$. Then,
	\[ 
		\log\Big\|\prod_{n=2^{2j-1}}^{2^{2j}-1} M_n(z)\Big\|\leq 
		C_\epsilon|z|^{1/(\beta+\gamma)+\epsilon}\sum_{n=2^{2j-1}}^{2^{2j}}\frac 1{n^{1+\epsilon}}
		.
	\]
	The proof of this fact can actually be taken verbatim from the proof of the relevant part of Theorem~1 in \cite{romanov:20XX} 
	by choosing the $a_n^2$ as in the proof of \thref{M2} above (in the case $\Delta_l>1$, $\Delta_l+\Delta_\phi>2$). 
	Adding up the obtained estimates in $j$ we find that the product 
	\[ 
		\prod_{j=1}^\infty \Big\|\prod_{n=2^{2j-1}}^{2^{2j}-1} M_n(z)\Big\|
	\]
	is estimated above by $e^{C|z|^{1/(\beta+\gamma)+\epsilon}}$. By the chain rule for the monodromy matrix it follows that 
	$\rho(H)\leq\max\{\alpha^{-1},(\beta+\gamma)^{-1}\}$. Now the condition on parameters $\alpha$, $\beta$ and $\gamma$ 
	in the assumption ensures that 
	\[
		\frac 1\alpha<\delta_l+\delta_\phi\quad\text{and}\quad\frac 1{\beta+\gamma}<\delta_l+\delta_\phi
		.
	\]
\end{proof}

\subsection{Regularly distributed data}

The formula for $\rho(H)$ given in \thref{R24} below is obtained by comparing the upper and lower bounds of
\thref{M2} and \thref{R2}. The decisive property which enables to show that these bounds coincide is a certain 
regularity of the distribution of the sequences of lengths and angle-differences. 

\begin{definition}\thlab{R21}
	We call a sequence $\vec y=(y_n)_{n=1}^\infty$ of positive real numbers \emph{regularly distributed}, if
	\[
		\frac{y_n}{\Big(\prod\limits_{k=1}^n y_k \Big)^{\frac{1}{n}}}={\rm O}(1).
	\]
\end{definition}

\noindent
This notion of regularity rules out heavy oscillations but also sparse peaks where very large or very small elements occur. 

\begin{remark}\thlab{R34}
	Many examples of regularly distributed sequences are provided by the following observations.
	\begin{enumerate}[$(i)$]
	\item Each monotonically decreasing sequence is regularly distributed. 
	\item If $\vec y$ is regularly distributed and $u_n\asymp y_n$, then $\vec u$ is regularly distributed. 
	\end{enumerate}
\end{remark}

\begin{lemma}\thlab{R30}
	Let $\vec y=(y_n)_{n=1}^\infty$ be a bounded and regularly distributed sequence of positive real numbers and let $\alpha\geq 0$. 
	Then 
	\[
		\Delta^*(\vec y)=\Delta(\vec y)=\delta(\vec y,\alpha)
		.
	\]
\end{lemma}
\begin{proof}
	By \thref{R47} it remains to show that $\delta(\vec y,\alpha)\leq\Delta^*(\vec y)$. 
	If $\delta(\vec y,\alpha)=0$, this inequality holds trivially. 
	Assume that $\delta(\vec y,\alpha)$ is positive, and consider $\tau\in[0,\delta(\vec y,\alpha))$. 
	For all sufficiently large $m$ we have 
	\[
		\tau\leq G(m;\vec y,\alpha))=\frac{-1}{m\ln m}\ln\Big(y_m^\alpha\prod_{k=1}^{m-1}y_k\Big)
		.
	\]
	This implies that 
	\[
		m^{-\tau}\geq\Big(y_m^\alpha\prod_{k=1}^{m-1}y_k\Big)^{\frac 1m}=
		y_m^{\frac{\alpha-1}m}\Big(\prod_{k=1}^my_k\Big)^{\frac 1m}
		\gtrsim b^{\frac{\alpha-1}m}y_m
		,
	\]
	and we conclude that $y_n\lesssim n^{-\tau}$. 
\end{proof}

\noindent
In the formulation of the next result the quantity $ \Lambda $ from  \thref{M1} is used. 

\begin{theorem}\thlab{R24}
	Let $H$ be a limit circle Hamburger Hamiltonian with lengths $\vec l$ and angles $\vec\phi$. 
	Assume that $\vec l$ is regularly distributed, that at least one of $\delta_l$ and $\delta_\phi$ exists as a limit, 
	and that either 
	\begin{itemize}
	\item[{\rm(A)}] The sequence $(|\sin(\phi_{n+1}-\phi_n)|)_{n=1}^\infty$ is regularly distributed, 
		$\delta_l+\delta_\phi\geq 2$, and $(\delta_l,\delta_\phi,\Lambda)\neq(1,1,0)$, 
	\end{itemize}
	or
	\begin{itemize}
	\item[{\rm(B)}] $\delta_\phi=0$.
	\end{itemize}
	Then
	\[
		\rho(H)=\frac 1{\delta_l+\delta_\phi}
		.
	\]
\end{theorem}	
\begin{proof}
	The sequence $\vec l$ is summable, and hence $\Delta(\vec l)\geq 1$. 
	\thref{R30} gives 
	\[
		\delta_l=\delta(\vec l,\frac 12)=\Delta(\vec l)=\Delta_l^+
		.
	\]
	Moreover, the overall assumptions of the theorem ensure that $\delta_{l,\phi}=\delta_l+\delta_\phi$. 
	Assume that {\rm(A)} holds. Then 
	\[
		\delta_\phi=\delta\big((|\sin(\phi_{n+1}-\phi_n)|)_{n=1}^\infty,0\big)=\Delta((|\sin(\phi_{n+1}-\phi_n)|)_{n=1}^\infty)
		=\Delta_\phi
		,
	\]
	and the further assumption in {\rm(A)} just say that \thref{M2} is applicable and that we are in the generic region. 
	Thus 
	\[
		\frac 1{\Delta_l^++\Delta_\phi}=\frac 1{\delta_l+\delta_\phi}=\frac 1{\delta_{l,\phi}}
		\leq\rho(H)\leq\frac 1{\Delta_l^++\Delta_\phi}
		.
	\]
	If {\rm(B)} holds, it is enough to remember \thref{R27} and \thref{R2} which yield 
	\[
		\frac 1{\Delta_l^+}=\frac 1{\delta_l}=\frac 1{\delta_{l,\phi}}\leq\rho(H)\leq\frac 1{\Delta_l^+}
		.
	\]
\end{proof}

\noindent
Let us now discuss two basic examples. 

\begin{example}[Power-like behaviour]\thlab{R4}
	Let $H$ be a Hamburger Hamiltonian with lengths $\vec l$ and angles $\vec\phi$, and assume that 
	\[
		l_n\asymp n^{-\alpha},\quad |\sin(\phi_{n+1}-\phi_n)|\asymp n^{-\beta}
		,
	\]
	with $\alpha>1$ and $\beta\geq 0$. 
	Then both sequences $\vec l$ and $(|\sin(\phi_{n+1}-\phi_n)|)_{n=1}^\infty$ are regularly distributed, cf.\ \thref{R34}, 
	$\delta_l$ and $\delta_\phi$ exist as limits, and 
	\[
		\delta_l=\Delta_l^+=\alpha,\quad \delta_\phi=\Delta_\phi=\Delta_\phi^*=\beta
		.
	\]
	If $\alpha+\beta\geq 2$, we obtain 
	\[
		\rho(H)=\frac 1{\alpha+\beta}
		,
	\]
	if $\alpha+\beta<2$, we have the bounds  
	\[
		\frac 1{\alpha+\beta}\leq\rho(H)\leq \frac{1-\beta}{\alpha-\beta}
		.
	\]
\end{example}

\begin{example}[jumping angles]\thlab{R5}
	Consider a limit circle Hamburger Hamiltonian whose angles $\vec\phi$ satisfy $|\sin(\phi_{n+1}-\phi_n)|\asymp 1$. 
	Then $\delta_\phi=\Delta_\phi=\Delta_\phi^*=0$, and $\delta_\phi$ exists as a limit. \thref{R27} and \thref{R2} 
	imply 
	\[
		\frac 1{\delta_l}\leq\rho(H)\leq\frac 1{\Delta_l^+}
		,
	\]
	remember here \eqref{R28}. 
	Using the more involved \thref{M2} inside the critical triangle does not improve this bound. 
	Regardless of the value of $\Lambda$, the maximum in \thref{M2}, $(ii)$, equals $\frac 1{\Delta_l^+}$. 
\end{example}

\section{Diagonal Hamiltonians with irregularly distributed lengths and the Liv\v{s}ic estimate}

If the lengths and angle-differences of a Hamburger Hamiltonian are not regularly distributed, the 
upper and lower bounds from \thref{R27}, \thref{M2} and \thref{R2} need not coincide. 
This section is devoted to the construction of -- simple and explicit -- examples which show that neither of these bounds 
necessarily coincides with the order. 

Such examles are already found in the class of \emph{diagonal} Hamburger Hamiltonians, i.e., 
Hamburger Hamiltonian with angles $\phi_n$ all being integer multiples of $\frac\pi 2$. 
To make the connection with moment problems, by \eqref{R56}, diagonal Hamburger Hamiltonians correspond to moment 
sequences with $s_n=0$ for all odd $n$. In turn, these are the Hamburger moment problems which arise from symmetrising 
a Stieltjes moment problem. 

Remember \thref{R5} which shows in particular that for a diagonal Hamburger Hamiltonian $H$ always 
$\delta_\phi(H)=\Delta_\phi(H)=\Delta_\phi^*(H)=0$ and so 
\begin{equation}\label{R60}
	\frac 1{\delta_l(H)}\leq\rho(H)\leq \inf \{ p > 0 \colon \vec l \in \ell^p \} \le \frac 1{\Delta_l^+(H)}
	.
\end{equation}

\begin{theorem}\thlab{M56}
	Let $\alpha\in[1,\infty)$ and $\beta\in(\alpha,\infty)$ be arbitrarily prescribed numbers. Let $q\in\bb N$, $q\geq 2$, 
	and consider the Hamburger Hamiltonian $H_q$ with angles $\phi_n:=n\frac\pi2$, $n\in\bb N$, and lengths 
	\[
		l_n:=
		\begin{cases}
			1 &\hspace*{-3mm},\quad n=1,
			\\
			(n\ln^2 n)^{-\alpha} &\hspace*{-3mm},\quad n\not\equiv 0\mod q,\ n\geq 2,
			\\
			(n\ln^2 n )^{-\beta} &\hspace*{-3mm},\quad n\equiv 0\mod q.
		\end{cases}
	\]
	Then 
	\begin{equation}\label{R59}
		\inf \{ p > 0 \colon \vec l \in \ell^p \} = \frac 1\alpha ,\quad \delta_l(H_q)=\frac{q-1}q\alpha+\frac 1q\beta,
	\end{equation}
	and 
	\[
		\rho(H_q)=
		\begin{cases}
			\frac 1{\delta_l(H_q)} &\hspace*{-3mm},\quad q=2,
			\\
			\frac 1\alpha &\hspace*{-3mm},\quad q\geq 3.
		\end{cases}
	\]
\end{theorem}

\noindent
On first sight it may seem peculiar that $\rho(H_q)$ is different for $q=2$ and $q=3$, but constant for $q\geq 3$.
One intuitive explanation might be that the dominating subsequence of lengths ($n\not\equiv 0\mod q$) sees 
the jumps of angles if $q\geq 3$ whereas for $q=2$ it does not.

\begin{proof}[Proof of \thref{M56}; the equalities \eqref{R59}]
	The first equality in \eqref{R59} is obvious. In order to compute $\delta_l(H_q)$, consider first 
	the sequence $\vec\lambda$ defined by 
	\[
		\lambda_n:=
		\begin{cases}
			n^{-\alpha} &\hspace*{-3mm},\quad n\not\equiv 0\mod q,
			\\
			n^{-\beta} &\hspace*{-3mm},\quad n\equiv 0\mod q.
		\end{cases}
	\]
	Then 
	\begin{multline*}
		\prod_{k=1}^n\lambda_k=\prod_{k=1}^nk^{-\alpha}\cdot\Big(\prod_{j=1}^{\lfloor\frac nq\rfloor}(qj)^\alpha\Big)
		\cdot\Big(\prod_{j=1}^{\lfloor\frac nq\rfloor}(qj)^{-\beta}\Big)
		\\
		=(n!)^{-\alpha}\cdot q^{\alpha\lfloor\frac nq\rfloor}\big(\lfloor{\textstyle\frac nq\rfloor}!\big)^\alpha\cdot
		q^{-\beta\lfloor\frac nq\rfloor}\big(\lfloor{\textstyle\frac nq\rfloor}!\big)^{-\beta}
		.
	\end{multline*}
	Applying the Stirling formula we find that
	\[
		\lim_{n\to\infty}G(n;\vec\lambda,{\textstyle\frac 12})=\frac{q-1}q\alpha+\frac 1q\beta
		.
	\]
	It holds that $\lim_{n\to\infty}G(n;(\frac{l_k}{\lambda_k})_{k=1}^\infty,\frac 12)=0$, and we obtain 
	\[
		\delta_l(H_q)=\lim_{n\to\infty}G(n;\vec l,\frac 12)=\lim_{n\to\infty}G(n;\vec\lambda,\frac 12)=\frac{q-1}q\alpha+\frac 1q\beta
		.
	\]
\end{proof}

\noindent
Computing the order of $H_q$ for $q\geq 3$ relies on \cite[Theorem~2]{romanov:20XX} which tells us how to compute the order of a 
diagonal Hamiltonian. 
Let us recall this theorem. To formulate it, one more notation is needed. 

\begin{definition}\thlab{M27}
	Consider a nonempty interval $[a,b)$. We denote by $\Cov[a,b)$ the set of all coverings $\Omega$ of $[a,b)$ 
	by finitely many pairwise disjoint left-closed and right-open intervals contained in $[a,b)$. 
\end{definition}

\noindent
Moreover, we denote by $\lambda$ the Lebesgue-measure on $\bb R$ and by $\# F$ the number of elements of a finite set $F$. 

\begin{theorem}[\cite{romanov:20XX}]\thlab{M5}
	Let $L\in(0,\infty)$, and let $H:[0,L)\to\bb R^{2\!\times\!2}$ be a Hamiltonian with $\tr H=1$ a.e.\ and 
	\begin{equation}\label{M44}
		\det H(x)=0,H(x)\text{ diagonal},\quad x\in[0,L)\text{ a.e.}
	\end{equation}
	Set 
	\begin{equation}\label{M41}
	\begin{aligned}
		M_1&\,:=\Big\{x\in[0,L):H(x)={\scriptsize\begin{pmatrix} 1 &\hspace*{-5pt} 0\\ 0 &\hspace*{-5pt} 0\end{pmatrix}}\Big\},
		\\
		M_2&\,:=\Big\{x\in[0,L):H(x)={\scriptsize\begin{pmatrix} 0 &\hspace*{-5pt} 0\\ 0 &\hspace*{-5pt} 1\end{pmatrix}}\Big\}.
	\end{aligned}
	\end{equation}
	Then $\rho(H)$ is equal to the infimum of all numbers $d\in(0,1]$ for which there exists a family $(\Omega(R))_{R>1}$ 
	of coverings $\Omega(R)\in\Cov[0,L)$ such that 
	\begin{equation}\label{M47}
		\#\Omega(R) = {\rm O}\big(R^d\big)
		,
	\end{equation}
	\begin{equation}\label{M46}
		\sum_{\omega\in\Omega(R)}\sqrt{\lambda(\omega\cap M_1)\cdot\lambda(\omega\cap M_2)} = {\rm O}\big(R^{d-1}\big)
		.
	\end{equation}
\end{theorem}

\noindent
Observe that, since the real and symmetric $2\!\times\!2$-matrix $H(x)$ satisfies $\tr H(x)=1$, it holds that $\det H(x)=0$ 
and $H(x)$ is diagonal if and only if 
\[
	H(x)=\begin{pmatrix} 1 &\hspace*{-5pt} 0\\ 0 &\hspace*{-5pt} 0\end{pmatrix}
	\quad\text{or}\quad
	H(x)=\begin{pmatrix} 0 &\hspace*{-5pt} 0\\ 0 &\hspace*{-5pt} 1\end{pmatrix}
	.
\]
It is an important observation that in \thref{M5} it suffices to consider coverings by intervals with endpoints at nodes $x_n$. 

\begin{definition}\thlab{M30}
	Let $H$ be a diagonal Hamburger Hamiltonian, and let $x_n$, $n=0,1,\ldots,\infty$ be 
	as in \eqref{M32}. We write $\Omega\in\Cov(H)$, if 
	\begin{enumerate}[$(i)$]
	\item $\Omega\in\Cov[0,x_\infty)$,
	\item $\forall\omega\in\Omega\,\exists n_-,n_+\in\bb N_0\cup\{\infty\}:\ \omega=[x_{n_-},x_{n_+})$. 
	\end{enumerate}
\end{definition}

\begin{lemma}\thlab{M31}
	Let $H$ be a diagonal limit circle Hamburger Hamiltonian. 
	Then $\rho(H)$ is equal to the infimum of all numbers $d\in(0,1]$ for which there exists a family $(\Omega(R))_{R>1}$ 
	of coverings $\Omega(R)\in\Cov(H)$ such that \eqref{M47} and \eqref{M46} hold. 
\end{lemma}
\begin{proof}
	It is enough to show that for each number $d\in(0,1]$ and family $(\Omega(R))_{R>1}$, $\Omega(R)\in\Cov[x_0,x_\infty)$, 
	with \eqref{M47} and \eqref{M46}, there exists a family $(\tilde\Omega(R))_{R>1}$, $\tilde\Omega(R)\in\Cov(H)$, 
	such that \eqref{M47} and \eqref{M46} still hold. 

	The coverings $\tilde\Omega(R)$ are constructed by modifying $\Omega(R)$ in the obvious way. Let $\omega\in\Omega(R)$. 
	\begin{list}{}{\leftmargin=0pt}
	\item --- \textit{Case~1:} 
		Assume that there exists an $n\in\bb N_0$ such that $\omega\subseteq[x_n,x_{n+1})$. Then we include the 
		interval $[x_n,x_{n+1})$ into $\tilde\Omega(R)$. 
	\item --- \textit{Case~2:} 
		Assume that Case~1 does not take place. Then there exists an $n\in\bb N$ such that $x_n$ lies in the 
		interior of $\omega$. Set 
		\begin{align*}
			& n_-:=\min\big\{n\in\bb N:x_n\text{ inner point of }\omega\big\},
			\\
			& n_+:=\max\big\{n\in\bb N:x_n\text{ inner point of }\omega\big\},
		\end{align*}
		and include the intervals (the middle interval appears only if $n_-<n_+$) 
		\[
			[x_{n_--1},x_{n_-}),\quad [x_{n_-},x_{n_+}),\quad [x_{n_+},x_{n_++1}) 
		\]
		into $\tilde\Omega(R)$. 
	\end{list}
	Then $\tilde\Omega(R)\in\Cov(H)$ and $\#\tilde\Omega(R)\leq 4\cdot\#\Omega(R)$. 
	In particular, \eqref{M47} holds for $(\tilde\Omega(R))_{R>1}$. 

	Consider the sum in \eqref{M46} for the covering $\tilde\Omega(R)$. Then only 
	intervals of the form $[x_{n_-},x_{n_+})$ constructed from some $\omega\in\Omega(R)$ contribute a possibly nonzero summand. 
	However, $[x_{n_-},x_{n_+})\subseteq\omega$ and hence 
	\[
		\lambda\big([x_{n_-},x_{n_+})\cap M_i\big)\leq\lambda(\omega\cap M_i),\quad i=1,2
		.
	\]
	We see that 
	\[
		\sum_{\tilde\omega\in\tilde\Omega(R)}\sqrt{\lambda(\tilde\omega\cap M_1)\cdot\lambda(\tilde\omega\cap M_2)}\leq 
		\sum_{\omega\in\Omega(R)}\sqrt{\lambda(\omega\cap M_1)\cdot\lambda(\omega\cap M_2)}
		,
	\]
	and conclude that \eqref{M46} holds. 
\end{proof}

\begin{proof}[Proof of \thref{M56}; computing $\rho(H_q)$]
	The estimate \eqref{R60} and \eqref{R59} give 
	\begin{equation}\label{R58}
		\Big[\frac{q-1}q\alpha+\frac 1q\beta\Big]^{-1}\leq\rho(H_q)\leq\frac 1\alpha
		.
	\end{equation}
	First, we consider the case that $q=2$. We are going to employ \cite[Theorem~1.2]{berg.szwarc:1509.06540v1}. 
	Since $H$ is diagonal, the orthogonal polynomials $P_n$ are even for even $n$ and odd for odd $n$, and the $Q_n$ 
	are odd for even $n$ and even for odd $n$. 
	Hence, 
	\[
		P_{2n}(0)^2=l_{2n},\ Q_{2n-1}(0)^2=l_{2n-1},\quad n\in\bb N
		,
	\]
	and hence $(P_{2n}(0)^2)_{n=1}^\infty\in\ell^{1/\beta}$ and $(Q_{2n-1}(0)^2)_{n=1}^\infty\in\ell^{1/\alpha}$.
	Moreover, both sequences are monotonically decreasing and 
	\[
		\frac{P_{2n}(0)^{2/\alpha}}{Q_{2n-1}(0)^{2/\beta}}=\frac{(2n-1)\ln^2(2n-1)}{2n\ln^2(2n)}\to 1
		.
	\]
	Hence \cite[Theorem~1.2]{berg.szwarc:1509.06540v1} is indeed applicable, and yields 
	$\rho(H_2)\leq[\frac 12(\alpha+\beta)]^{-1}$. Together with \eqref{R58}, thus $\rho(H_2)=\frac 2{\alpha+\beta}=\delta_l(H_2)^{-1}$. 

	Now assume that $q\geq 3$. In view of \eqref{R58} we have to show that $\rho(H_q)\geq\frac 1\alpha$. 
	To this end, consider the auxiliary diagonal Hamburger Hamiltonian $\tilde H$ 
	with lengths 
	\[
		h_n:=
		\begin{cases}
			1 &\hspace*{-3mm},\quad n=1,
			\\
			\big(n\ln^2 n\big)^{-\alpha},\quad n\in\bb N,\ n\geq 2,
		\end{cases}
	\]
	and the same angles as $H$. 
	By monotonicity, $\vec h$ is regularly distributed and \eqref{R60} gives $\rho(\tilde H)=\frac 1\alpha$. 

	Let $d>\rho(H_q )$ and choose, by virtue of \thref{M5}, 
	a family of coverings $(\Omega(R))_{R>1}\in\Cov(H)$ such that \eqref{M47} and \eqref{M46} hold for $d$. 
	We are going to modify this family so as to obtain a family of coverings for $\tilde H$. 
	First we refine the given coverings to construct $(\Omega'(R))_{R>1}\in\Cov(H)$ such that 
	\eqref{M47} and \eqref{M46} hold for $(\Omega'(R))_{R>1}$ and $d$, and such that: 
	\begin{equation}\label{R57}
		\parbox{78mm}{
		\emph{If $j\in\bb N$ and $\omega\in\Omega'(R)$ contains $[x_{qj-1},x_{qj})$, then
		either $\omega=[x_{qj-1},x_{qj})$ or $\omega\supseteq[x_{qj-3},x_{qj-2})$.}
		}
	\end{equation}
	Indeed, since $q\geq 3$, this property can be achieved by splitting the intervals $\omega\in\Omega(R)$ 
	in at most three smaller ones, namely by cutting off the first or the first two intervals of $H$ which lie in $\omega$ 
	if necessary, and adding them to $ \Omega^\prime ( R ) $. 
	We have $\#\Omega'(R)\leq 3\cdot \#\Omega(R)$ and the sum in \eqref{M46} for $\Omega'(R)$ does not exceed 
	the one for $\Omega(R)$. Hence \eqref{M47} and \eqref{M46} hold for $(\Omega'(R))_{R>1}$ and $d$. 

	The property \eqref{R57} and monotonicity of $(h_j)_{j=1}^\infty$ implies that 
	for each interval $\omega$ which is not equal to a single interval of $H$, 
	\[
		\sum_{\substack{n\equiv 0\mkern-12mu\mod q\\ [x_{n-1},x_n)\subseteq\omega\cap M_i}}\mkern-20mu h_n\quad 
		\leq\sum_{\substack{n\not\equiv 0\mkern-12mu\mod q\\ [x_{n-1},x_n)\subseteq\omega\cap M_i}}\mkern-20mu h_n
		\ \leq\ \lambda(\omega\cap M_i),\quad \omega\in\Omega'(R),\ i=1,2
		,
	\]
	and hence 
	\begin{equation}\label{R35}
		\sum_{\substack{n\in\bb N\\ [x_{n-1},x_n)\subseteq\omega\cap M_i}}\mkern-20mu h_n\ \leq\ 2\lambda(\omega\cap M_i)
		,\quad \omega\in\Omega'(R),\ i=1,2
		.
	\end{equation}
	Denote by $\tilde x_n$ the nodes of $\tilde H$, and define $\tilde\Omega(R)\in\Cov(\tilde H)$ to be the covering 
	\[
		\tilde\Omega(R):=\big\{[\tilde x_n,\tilde x_m):\ n,m\in\bb N,[x_n,x_m)\in\Omega'(R)\big\}
		.
	\]
	Then we have $\#\tilde\Omega(R)=\#\Omega(R)$ and, by \eqref{R35}, 
	\[
		\lambda([\tilde x_n,\tilde x_m)\cap\tilde M_i)\leq 2\lambda([x_n,x_m)\cap M_i)
		,\quad [\tilde x_n,\tilde x_m)\in\tilde\Omega(R),\ i=1,2
		.
	\]
	We see that $\tilde\Omega(R)$ satisfies \eqref{M47} and \eqref{M46}. 
	Referring again to \thref{M5}, this time for $ \tilde H $, gives $d\geq\rho(\tilde H)=\frac 1\alpha$, and the result follows. 
\end{proof}

\noindent
Finally, let us discuss the Liv\v{s}ic estimate \eqref{Livsic}. To translate \thref{M56} into this language, 
we recall the modern proof of \eqref{Livsic} given in \cite{berg.szwarc:2014}. 

\begin{proof}[Deduction of \eqref{Livsic} from \thref{R2}]
	Given a sequence $(s_n)_{n=0}^\infty$ of power moments of a measure $\mu$ on the real axis, we have for
	the corresponding orthogonal polynomials $P_n$ 
	\[
		1=(P_n,P_n)_{L^2(\mu)}=b_{n,n}(z^n,P_n)_{L^2(\mu)}\leq
		b_{n,n} \|z^n\|_{L^2(\mu)} = b_{ n , n } \sqrt{s_{2n}}
		,
	\]
	from whence 
	\begin{equation}\label{M62}
		b_{n,n}\geq\frac 1{\sqrt{s_{2n}}}
		.
	\end{equation}
	Using \thref{R2}, \eqref{R37}, and plugging \eqref{M62}, yields
	\begin{equation}\label{bergbound}
		\limsup_{n\to\infty}\frac{2n\ln n}{\ln s_{2n}}\leq\limsup_{n\to\infty}\frac{n\ln n}{\ln b_{n,n}^{-1}}
		\leq\rho\big((s_n)_{n=0}^\infty\big)
		.
	\end{equation}
\end{proof}

\noindent
From \thref{M56} we now obtain examples for which the second inequality in \eqref{bergbound} is strict. 

\begin{corollary}\thlab{Livcounter}
	For any $\rho\in(0,1]$ and $r\in(0,\rho)$ there exists an indeterminate moment sequence $(s_n)_{n=1}^\infty$ such that
	\[
		\rho\big((s_n)_{n=0}^\infty\big)=\rho\quad\text{and}\quad \limsup_{n\to\infty}\frac{n\ln n}{\ln b_{n,n}^{-1}}=r
		. 
	\]
	The sequence can be chosen so that $s_n=0$ for all odd $n$. 
\end{corollary}
\begin{proof}
	Let $q\in\bb N$, $q\geq 3$, and set $\alpha:=\frac 1\rho$ and $\beta:=\frac qr-(q-1)\alpha$. 
	Then the moment sequence corresponding to $H_q$ has all the required properties. 
\end{proof}

\noindent
We close the paper with formulating an open question. We have just established that the second inequality in 
\eqref{bergbound} may be strict. Are there moment problems for which  
\[
	\limsup_{n\to\infty}\frac{2n\ln n}{\ln s_{2n}}<\limsup_{n\to\infty}\frac{n\ln n}{\ln b_{n,n}^{-1}}\quad {\bf ?}
\]
The answer is expected to be affirmative.


{\footnotesize
\begin{flushleft}
	R.\,Pruckner\\
	Institute for Analysis and Scientific Computing\\
	Vienna University of Technology\\
	Wiedner Hauptstra{\ss}e\ 8--10/101\\
	1040 Wien\\
	AUSTRIA\\
	email: raphael.pruckner@tuwien.ac.at\\[5mm]
\end{flushleft}
}
{\footnotesize
\begin{flushleft}
	R.\,Romanov\\
	Department of Mathematical Physics and Laboratory of Quantum Networks\\
	Faculty of Physics, St Petersburg State University\\
	198504 St.Petersburg\\
	RUSSIA\\
	email: morovom@gmail.com\\[5mm]
\end{flushleft}
}
{\footnotesize
\begin{flushleft}
	H.\,Woracek\\
	Institute for Analysis and Scientific Computing\\
	Vienna University of Technology\\
	Wiedner Hauptstra{\ss}e\ 8--10/101\\
	1040 Wien\\
	AUSTRIA\\
	email: harald.woracek@tuwien.ac.at\\[5mm]
\end{flushleft}
}


\ifthenelse{\Draft=1}{
\newpage

\noindent
{\large\sc labels:}\\[10mm]
	
M:\hspace*{3mm}
\framebox{
\begin{tabular}{r@{\ }r@{\ }r@{\ }r@{\ }r@{\ }r@{\ }r@{\ }r@{\ }r@{\ }r@{\ }}
	** &  1 &  2 &  3 &  4 &  5 &  6 &  7 &  8 &  9 \\
	10 & 11 & 12 & 13 & 14 & 15 & 16 & 17 & 18 & 19 \\
	20 & 21 & 22 & 23 & 24 & 25 & 26 & 27 & 28 & 29 \\
	30 & 31 & 32 & 33 & 34 & 35 & 36 & 37 & 38 & 39 \\
	40 & 41 & 42 & 43 & 44 & 45 & 46 & 47 & 48 & 49 \\
	50 & 51 & 52 & 53 & 54 & 55 & 56 & 57 & 58 & 59 \\
	60 & 61 & 62 & 63 & 64 & 65 & 66 & 67 & 68 & 69 \\
	70 & 71 & 72 & 73 & 74 & 75 & 76 & 77 & 78 & 79 \\
	80 & 81 & 82 & 83 & 84 & 85 & 86 & 87 & 88 & 89 \\
	 . &  . &  . &  . &  . &  . &  . &  . &  . &  . \\
\end{tabular}
}
\\[1cm]

R:\hspace*{3mm}
\framebox{
\begin{tabular}{r@{\ }r@{\ }r@{\ }r@{\ }r@{\ }r@{\ }r@{\ }r@{\ }r@{\ }r@{\ }}
	** &  1 &  2 &  3 &  4 &  5 &  6 &  7 &  8 &  9 \\
	10 & 11 & 12 & 13 & 14 & 15 & 16 & 17 & 18 & 19 \\
	20 & 21 & 22 & 23 & 24 & 25 & 26 & 27 & 28 & 29 \\
	30 & 31 & 32 & 33 & 34 & 35 & 36 & 37 & 38 & 39 \\
	40 & 41 & 42 & 43 & 44 & 45 & 46 & 47 & 48 & 49 \\
	50 & 51 & 52 & 53 & 54 & 55 & 56 & 57 & 58 & 59 \\
	60 & 61 &  . &  . &  . &  . &  . &  . &  . &  . \\
	 . &  . &  . &  . &  . &  . &  . &  . &  . &  . \\
	 . &  . &  . &  . &  . &  . &  . &  . &  . &  . \\
\end{tabular}
}
\\[1cm]
}
{}

\end{document}